\documentclass[12pt,reqno]{amsart} 
\usepackage{amssymb,amscd,url,amsmath}
\usepackage{xcolor}    
\usepackage[all]{xy}  
\usepackage{bbm}      
\usepackage{pict2e}
\usepackage{calligra}  
\usepackage{constants} 

\usepackage{mathtools}

\newconstantfamily{DZ}{
  symbol = {\mathcal{C}},
  format = \arabic
  }

\begin{document}

\allowdisplaybreaks


\title[Lehmer Bound for Abelian Surfaces]
      {A Lehmer-Type Height Lower Bound for Abelian Surfaces over Function Fields}
\date{\today}
\author[N.R. Looper]{Nicole R. Looper}
\email{nicole\_looper@brown.edu}
\address{Department of Mathematics, Box 1917
  Brown University, Providence, RI 02912 USA}
\author[J.H. Silverman]{Joseph H. Silverman}
\email{joseph\_silverman@brown.edu}
\address{Department of Mathematics, Box 1917
  Brown University, Providence, RI 02912 USA.
  ORCID: https://orcid.org/0000-0003-3887-3248}

\thanks{The first author was supported by NSF grant DMS-1803021. The
  second author was partially supported by Simons Collaboration Grant
  $\# 712332$.}



\newcommand{\JOE}[1]{{\color{blue} $\bigstar$ \textsf{{\bf Joe:} [#1]}}}
\newcommand{\NICOLE}[1]{{\color{red} $\bigstar$ \textsf{{\bf Nicole:} [#1]}}}

\hyphenation{ca-non-i-cal semi-abel-ian abel-ian}


\newtheorem{theorem}{Theorem}[section]
\newtheorem{lemma}[theorem]{Lemma}
\newtheorem{sublemma}[theorem]{Sublemma}
\newtheorem{conjecture}[theorem]{Conjecture}
\newtheorem{proposition}[theorem]{Proposition}
\newtheorem{corollary}[theorem]{Corollary}

\theoremstyle{definition}
\newtheorem*{claim}{Claim}
\newtheorem{definition}[theorem]{Definition}
\newtheorem*{intuition}{Intuition}
\newtheorem{example}[theorem]{Example}
\newtheorem{remark}[theorem]{Remark}
\newtheorem{question}[theorem]{Question}

\theoremstyle{remark}
\newtheorem*{acknowledgement}{Acknowledgements}


\newenvironment{notation}[0]{%
  \begin{list}%
    {}%
    {\setlength{\itemindent}{0pt}
     \setlength{\labelwidth}{4\parindent}
     \setlength{\labelsep}{\parindent}
     \setlength{\leftmargin}{5\parindent}
     \setlength{\itemsep}{0pt}
     }%
   }%
  {\end{list}}

\newenvironment{parts}[0]{%
  \begin{list}{}%
    {\setlength{\itemindent}{0pt}
     \setlength{\labelwidth}{1.5\parindent}
     \setlength{\labelsep}{.5\parindent}
     \setlength{\leftmargin}{2\parindent}
     \setlength{\itemsep}{0pt}
     }%
   }%
  {\end{list}}
\newcommand{\Part}[1]{\item[\upshape#1]}

\def\Case#1#2{%
\paragraph{\textbf{\boldmath Case #1: #2.}}\hfil\break\ignorespaces}

\newcommand{\Fhat}{\hat{F}}
\newcommand{\fhat}{\hat{f}}
\newcommand\FB[1]{\par\noindent\framebox{$\boldsymbol{#1}$}\enspace\ignorespaces}

\renewcommand{\a}{\alpha}
\newcommand{\bfalpha}{{\boldsymbol{\alpha}}}
\renewcommand{\b}{\beta}
\newcommand{\bfbeta}{{\boldsymbol{\beta}}}
\newcommand{\g}{\gamma}
\renewcommand{\d}{\delta}
\newcommand{\e}{\epsilon}
\newcommand{\f}{\varphi}
\newcommand{\bfphi}{{\boldsymbol{\f}}}
\renewcommand{\l}{\lambda}
\renewcommand{\k}{\kappa}
\newcommand{\lhat}{\hat\lambda}
\newcommand{\m}{\mu}
\newcommand{\bfmu}{{\boldsymbol{\mu}}}
\renewcommand{\o}{\omega}
\newcommand{\bfpi}{{\boldsymbol{\pi}}}
\renewcommand{\r}{\rho}
\newcommand{\bfrho}{{\boldsymbol{\rho}}}
\newcommand{\rbar}{{\bar\rho}}
\newcommand{\s}{\sigma}
\newcommand{\sbar}{{\bar\sigma}}
\renewcommand{\t}{\tau}
\newcommand{\z}{\zeta}

\newcommand{\D}{\Delta}
\newcommand{\G}{\Gamma}
\newcommand{\F}{\Phi}
\renewcommand{\L}{\Lambda}


\newcommand{\BarIt}[1]{{{\boldsymbol{\bar{#1}}}}}

\newcommand{\Kc}{\BarIt{\mathsf{K}}}
\newcommand{\Gc}{\BarIt\Gamma}
\newcommand{\Jc}{\BarIt{J}}
\newcommand{\Pc}{\BarIt{P}}
\newcommand{\Qc}{\BarIt{Q}}
\newcommand{\fc}{\BarIt{f}}
\newcommand{\Oc}{\BarIt{\Ocal}}

\newcommand{\ga}{{\mathfrak{a}}}
\newcommand{\gA}{{\mathfrak{A}}}
\newcommand{\gb}{{\mathfrak{b}}}
\newcommand{\gB}{{\mathfrak{B}}}
\newcommand{\gc}{{\mathfrak{c}}}
\newcommand{\gd}{{\mathfrak{d}}}
\newcommand{\gD}{{\mathfrak{D}}}
\newcommand{\gM}{{\mathfrak{M}}}
\newcommand{\gN}{{\mathfrak{N}}}
\newcommand{\gn}{{\mathfrak{n}}}
\newcommand{\gp}{{\mathfrak{p}}}
\newcommand{\gP}{{\mathfrak{P}}}
\newcommand{\gS}{{\mathfrak{S}}}
\newcommand{\gq}{{\mathfrak{q}}}

\newcommand{\Abar}{{\bar A}}
\newcommand{\Ebar}{{\bar E}}
\newcommand{\kbar}{{\bar k}}
\newcommand{\KCbar}{\overline{K}_C}
\newcommand{\Kbar}{{\bar K}}
\newcommand{\Pbar}{{\bar P}}
\newcommand{\Sbar}{{\bar S}}
\newcommand{\Tbar}{{\bar T}}
\newcommand{\gbar}{{\bar\gamma}}
\newcommand{\lbar}{{\bar\lambda}}
\newcommand{\ybar}{{\bar y}}
\newcommand{\phibar}{{\bar\f}}
\newcommand{\nubar}{{\overline\nu}}

\newcommand{\Acal}{{\mathcal A}}
\newcommand{\Bcal}{{\mathcal B}}
\newcommand{\Ccal}{{\mathcal C}}
\newcommand{\Dcal}{{\mathcal D}}
\newcommand{\Ecal}{{\mathcal E}}
\newcommand{\Fcal}{{\mathcal F}}
\newcommand{\Gcal}{{\mathcal G}}
\newcommand{\Hcal}{{\mathcal H}}
\newcommand{\Ical}{{\mathcal I}}
\newcommand{\Jcal}{{\mathcal J}}
\newcommand{\Kcal}{{\mathcal K}}
\newcommand{\Lcal}{{\mathcal L}}
\newcommand{\Mcal}{{\mathcal M}}
\newcommand{\Ncal}{{\mathcal N}}
\newcommand{\Ocal}{{\mathcal O}}
\newcommand{\Pcal}{{\mathcal P}}
\newcommand{\Qcal}{{\mathcal Q}}
\newcommand{\Rcal}{{\mathcal R}}
\newcommand{\Scal}{{\mathcal S}}
\newcommand{\Tcal}{{\mathcal T}}
\newcommand{\Ucal}{{\mathcal U}}
\newcommand{\Vcal}{{\mathcal V}}
\newcommand{\Wcal}{{\mathcal W}}
\newcommand{\Xcal}{{\mathcal X}}
\newcommand{\Ycal}{{\mathcal Y}}
\newcommand{\Zcal}{{\mathcal Z}}

\renewcommand{\AA}{\mathbb{A}}
\newcommand{\BB}{\mathbb{B}}
\newcommand{\CC}{\mathbb{C}}
\newcommand{\FF}{\mathbb{F}}
\newcommand{\GG}{\mathbb{G}}
\newcommand{\NN}{\mathbb{N}}
\newcommand{\PP}{\mathbb{P}}
\newcommand{\QQ}{\mathbb{Q}}
\newcommand{\RR}{\mathbb{R}}
\newcommand{\TT}{\mathbb{T}}
\newcommand{\ZZ}{\mathbb{Z}}

\newcommand{\bfa}{{\boldsymbol a}}
\newcommand{\bfb}{{\boldsymbol b}}
\newcommand{\bfc}{{\boldsymbol c}}
\newcommand{\bfd}{{\boldsymbol d}}
\newcommand{\bfe}{{\boldsymbol e}}
\newcommand{\ee}{{\boldsymbol{e}}} 
\newcommand{\bff}{{\boldsymbol f}}
\newcommand{\bfg}{{\boldsymbol g}}
\newcommand{\bfi}{{\boldsymbol i}}
\newcommand{\bfj}{{\boldsymbol j}}
\newcommand{\bfk}{{\boldsymbol k}}
\newcommand{\bfm}{{\boldsymbol m}}
\newcommand{\bfn}{{\boldsymbol n}}
\newcommand{\bfp}{{\boldsymbol p}}
\newcommand{\bfq}{{\boldsymbol q}}
\newcommand{\bfr}{{\boldsymbol r}}
\newcommand{\bfs}{{\boldsymbol s}}
\newcommand{\bft}{{\boldsymbol t}}
\newcommand{\bfu}{{\boldsymbol u}}
\newcommand{\bfv}{{\boldsymbol v}}
\newcommand{\bfw}{{\boldsymbol w}}
\newcommand{\bfx}{{\boldsymbol x}}
\newcommand{\bfy}{{\boldsymbol y}}
\newcommand{\bfz}{{\boldsymbol z}}
\newcommand{\bfA}{{\boldsymbol A}}
\newcommand{\bfF}{{\boldsymbol F}}
\newcommand{\bfB}{{\boldsymbol B}}
\newcommand{\bfD}{{\boldsymbol D}}
\newcommand{\bfE}{{\boldsymbol E}}
\newcommand{\bfG}{{\boldsymbol G}}
\newcommand{\bfI}{{\boldsymbol I}}
\newcommand{\bfM}{{\boldsymbol M}}
\newcommand{\bfP}{{\boldsymbol P}}
\newcommand{\bfQ}{{\boldsymbol Q}}
\newcommand{\bfS}{{\boldsymbol S}}
\newcommand{\bfT}{{\boldsymbol T}}
\newcommand{\bfU}{{\boldsymbol U}}
\newcommand{\bfX}{{\boldsymbol X}}
\newcommand{\bfY}{{\boldsymbol Y}}
\newcommand{\bfzero}{{\boldsymbol{0}}}
\newcommand{\bfone}{{\boldsymbol{1}}}

\newcommand{\ab}{{\textup{ab}}}
\newcommand{\abcFunctionOne}{\operatorname{\xi}}
\newcommand{\abcFunctionThree}{\operatorname{\Delta}}
\newcommand{\ASD}{\Delta} 
\newcommand{\Aut}{\operatorname{Aut}}
\newcommand{\Avg}{\operatorname{\hbox{\normalfont\calligra{Avg}}}}
\newcommand{\AvgL}{\operatornamewithlimits{\hbox{\normalfont\calligra{Avg}}}}
\newcommand{\Berk}{{\textup{Berk}}}
\newcommand{\adj}{{\textup{adj}}}
\newcommand{\Bern}{{\mathbb{B}}}   
\newcommand{\bad}{{\textup{bad}}}
\newcommand{\Birat}{\operatorname{Birat}}
\newcommand{\characteristic}{\operatorname{char}}
\newcommand{\codim}{\operatorname{codim}}
\newcommand{\Crit}{\operatorname{Crit}}
\newcommand{\crit}{{\textup{crit}}}
\newcommand{\critwt}{\operatorname{critwt}} 
\newcommand{\Cycle}{\operatorname{Cycles}}
\newcommand{\DD}[1]{\mathcal{D}_{(#1)}} 
\newcommand{\diag}{\operatorname{diag}}
\newcommand{\dimEnd}{{M}}  
\newcommand{\Disc}{\operatorname{Disc}}
\newcommand{\Div}{\operatorname{Div}}
\renewcommand{\div}{\operatorname{div}}
\newcommand{\Df}{{Df}}  
\newcommand{\DI}{\int\!\!\!\!\int}   
\newcommand{\Dom}{\operatorname{Dom}}
\newcommand{\dyn}{{\textup{dyn}}}
\newcommand{\End}{\operatorname{End}}
\newcommand{\PortEndPt}{{\textup{endpt}}} 
\newcommand{\END}{\smash[t]{\overline{\operatorname{End}}}\vphantom{E}}
\newcommand{\EndPoint}{E}  
\newcommand{\ExtOrbit}{\mathcal{EO}} 
\newcommand{\Fbar}{{\bar{F}}}
\newcommand{\FC}{{\widehat{L}}}  
\newcommand{\fib}{{\textup{fib}}}
\newcommand{\Fix}{\operatorname{Fix}}
\newcommand{\Fiber}{\operatorname{Fiber}}
\newcommand{\FOD}{\operatorname{FOD}}
\newcommand{\FOM}{\operatorname{FOM}}
\newcommand{\Frame}{\operatorname{Fr}}
\newcommand{\Gal}{\operatorname{Gal}}
\newcommand{\GITQuot}{/\!/}
\newcommand{\GL}{\operatorname{GL}}
\newcommand{\GR}{\operatorname{\mathcal{G\!R}}}
\newcommand{\hhat}{{\hat h}}
\newcommand{\Hom}{\operatorname{Hom}}
\newcommand{\Index}{\operatorname{Index}}
\newcommand{\Image}{\operatorname{Image}}
\newcommand{\Int}{{\mathbb{I}}}   
\newcommand{\Intersect}{{\textup{Intersect}}}
\newcommand{\Isom}{\operatorname{Isom}}
\newcommand{\Jac}{\operatorname{Jac}}
\newcommand{\Ker}{{\operatorname{Ker}}}
\newcommand{\Ksep}{K^{\text{sep}}}  
\newcommand{\LCM}{\operatorname{LCM}}
\newcommand{\Length}{\operatorname{Length}}
\newcommand{\Lift}{\operatorname{Lift}}
\newcommand{\limstar}{\lim\nolimits^*}
\newcommand{\limstarn}{\lim_{\hidewidth n\to\infty\hidewidth}{\!}^*{\,}}
\newcommand{\logplus}{\log^{\scriptscriptstyle+}}
\newcommand{\Mat}{\operatorname{Mat}}
\newcommand{\maxplus}{\operatornamewithlimits{\textup{max}^{\scriptscriptstyle+}}}
\newcommand{\MOD}[1]{~(\textup{mod}~#1)}
\newcommand{\Model}{\operatorname{Model}}
\newcommand{\Mor}{\operatorname{Mor}}
\newcommand{\Moduli}{\mathcal{M}}
\newcommand{\MODULI}{\overline{\mathcal{M}}}
\newcommand{\Mult}{\operatorname{\textup{\textsf{Mult}}}}
\newcommand{\Norm}{{\operatorname{\mathsf{N}}}}
\newcommand{\notdivide}{\nmid}
\newcommand{\normalsubgroup}{\triangleleft}
\newcommand{\NS}{\operatorname{NS}}
\newcommand{\onto}{\twoheadrightarrow}
\newcommand{\ord}{\operatorname{ord}}
\newcommand{\Orbit}{\mathcal{O}}
\newcommand{\Pcase}[3]{\par\noindent\framebox{$\boldsymbol{\Pcal_{#1,#2}}$}\enspace\ignorespaces}
\newcommand{\pd}{p}       
\newcommand{\bfpd}{\bfp}  
\newcommand{\Per}{\operatorname{Per}}
\newcommand{\Perp}{\operatorname{Perp}}
\newcommand{\PrePer}{\operatorname{PrePer}}
\newcommand{\PGL}{\operatorname{PGL}}
\newcommand{\Pic}{\operatorname{Pic}}
\newcommand{\Portrait}{\mathfrak{Port}}  
\newcommand{\prim}{\textup{prim}}
\newcommand{\Prob}{\operatorname{Prob}}
\newcommand{\Proj}{\operatorname{Proj}}
\newcommand{\Qbar}{{\bar{\QQ}}}
\newcommand{\rank}{\operatorname{rank}}
\newcommand{\Rat}{\operatorname{Rat}}
\newcommand{\rat}{{\textup{rat}}}
\newcommand{\Resultant}{\operatorname{Res}}
\newcommand{\Residue}{\operatorname{Residue}} 
\renewcommand{\setminus}{\smallsetminus}
\newcommand{\sgn}{\operatorname{sgn}}
\newcommand{\shafdim}{\operatorname{ShafDim}}
\newcommand{\SL}{\operatorname{SL}}
\newcommand{\Span}{\operatorname{Span}}
\newcommand{\Spec}{\operatorname{Spec}}
\renewcommand{\ss}{{\textup{ss}}}
\newcommand{\stab}{{\textup{stab}}}
\newcommand{\Stab}{\operatorname{Stab}}
\newcommand{\SemiStable}[1]{\textup{(SS$_{#1}$)}}  
\newcommand{\Stable}[1]{\textup{(St$_{#1}$)}}      
\newcommand{\Support}{\operatorname{Supp}}
\newcommand{\Sym}{\operatorname{Sym}}  
\newcommand{\Symp}{\operatorname{Sp}}  
\newcommand{\TableLoopSpacing}{{\vrule height 15pt depth 10pt width 0pt}} 
\newcommand{\ThetaFunction}{\theta}  
\newcommand{\ThetaDivisor}{\Theta} 
\newcommand{\ThetaDivisorT}{\Theta_0} 
\newcommand{\tors}{{\textup{tors}}}
\newcommand{\Trace}{\operatorname{Trace}}
\newcommand{\trianglebin}{\mathbin{\triangle}} 
\newcommand{\tr}{{\textup{tr}}} 
\newcommand{\UHP}{{\mathfrak{h}}}    
\newcommand{\val}{\operatorname{val}} 
\newcommand{\wt}{\operatorname{wt}} 
\newcommand{\<}{\langle}
\renewcommand{\>}{\rangle}

\newcommand{\pmodintext}[1]{~\textup{(mod}~#1\textup{)}}
\newcommand{\ds}{\displaystyle}
\newcommand{\longhookrightarrow}{\lhook\joinrel\longrightarrow}
\newcommand{\longonto}{\relbar\joinrel\twoheadrightarrow}
\newcommand{\SmallMatrix}[1]{%
  \left(\begin{smallmatrix} #1 \end{smallmatrix}\right)}

\newcommand{\COS}{\operatorname{\textbf{Cos}}}
\newcommand{\SIN}{\operatorname{\textbf{Sin}}}


\begin{abstract}
Let $K$ be a 1-dimensional function field over an algebraically closed
field of characteristic $0$, and let $A/K$ be an abelian
surface. Under mild assumptions, we prove a Lehmer-type lower bound
for points in~$A(\Kbar)$. More precisely, we prove that there are constants $C_1,C_2>0$ such that the
normalized Bernoulli-part of the canonical height is bounded below by
\[
\hat{h}_A^{\mathbb{B}}(P) \ge C_1\bigl[K(P):K\bigr]^{-2}
\]
for all points~$P\in{A(\Kbar)}$ whose height
satisfies~$0<\hat{h}_A(P)\le{C_2}$.
\end{abstract}

\maketitle

\tableofcontents



\section{Introduction}
\label{section:introduction}

The classical Lehmer conjecture says that there is an absolute
constant~$C>0$ such that if~$\a\in\Qbar^*$ is a not a root of unity,
then its absolute logarithmic height satisfies
\[
h(\a) \ge \frac{C}{\bigl[\QQ(\a):\QQ\bigr]}.
\]
Various authors have extended this conjecture to elliptic curves and to
higher dimensional abelian varieties. We review these conjectures and
some of the progress made in proving them in
Section~\ref{section:surveyolderresults}.

The main results of the present paper are: (1) an explicit Fourier
expansion of the ``Bernoulli-part'' of the canonical height on abelian
surfaces defined over non-archimedean local fields
(Theorem~\ref{theorem:fourierexpansionofL}); (2) a lower bound for a
torsion-and-difference average of the Bernoulli part of the height on
abelian surfaces defined over function fields
(Theorem~\ref{theorem:hlen23len23}). This is an analogue of the key
lemma in~\cite{hindrysilverman:lehmer}, which dealt with elliptic
curves. We also prove: (3) a Lehmer-type lower bound with exponent~$2$
for the canonical height of non-torsion points on abelian surfaces
defined over function fields that is conditional on the assumption
that the torsion-and-difference average of the ``intersection part''
of the canonical height is at least as large as a certain local-global
constant (Corollary~\ref{corollary:conditionallehmer}).

Before explaining the statements of our results in more detail, we
briefly recall the local decomposition of the canonical height on
abelian varieties.  See Section~\ref{section:overviewlocalhts} for
further details and references.

Let~$k$ be a algebraically closed field of characterstic~$0$,
let~$K/k$ be a $1$-dimensional function field, let~$A/K$ be an abelian
variety, and let~$\ASD\in\Div(A)$ be an ample symmetric divisor
on~$A$. The associated canonical height
\[
\hhat_{A,\ASD} : A(\Kbar) \longrightarrow \RR_{\ge0}
\]
may be decomposed as a sum of normalized local canonical heights
\[
\lhat_{A,\ASD,v} : \bigl(A\setminus|\ASD|\bigr)(\Kbar_v) \longrightarrow \RR,
\]
one for each absolute value on~$\Kbar$, where the normalization condition 
\begin{equation}
  \label{eqn:normcondlhatADv}
  \lim_{N\to\infty} \frac{1}{N^{2g}} \sum_{P\in A[N]\setminus|\ASD|} \lhat_{A,\ASD,v}(P) = 0
\end{equation}
serves to uniquely determine~$\lhat_{A,\ASD,v}$.  The local height
further decomposes into an ``intersection-part'' and a
``Bernoulli-part,'' which we denote respectively
by~$\lhat^\Int_{A,\ASD,v}$ and~$\lhat^\Bern_{A,\ASD,v}$.  The
intersection part is given by
\[
\lhat^\Int_{A,\ASD,v}(P) =
\left(
\parbox{.55\hsize}{
  intersection index of $\overline{P}$ and $\overline{\ASD}$
  on the $v$-fiber of the N\'eron model of~$A$
  } 
\right) - \kappa^\Int_{A,\ASD,v},
\]
where the constant~$\kappa^\Int_{A,\ASD,v}$ is chosen 
so
that~$\lhat_{A,\ASD,v}^\Int$ itself satisfies the normalization
condition~\eqref{eqn:normcondlhatADv}, and then the Bernoulli part
is what's left over, i.e., 
\[
\lhat_{A,\ASD,v}^{\vphantom{\Int}}(P) = \lhat^\Int_{A,\ASD,v}(P)+\lhat^\Bern_{A,\ASD,v}(P)
\]
We also note that 
\[
\text{$A$ has potential good reduction at $v$}
\quad\Longrightarrow\quad
\lhat^\Bern_{A,\ASD,v}=\kappa_{A,\ASD,v}^\Int=0.
\]
(See Section~\ref{section:overviewlocalhts} for the further details.)

For any finite extension~$L/K$, we write~$M_L$ for an appropriate
normalized set of absolute values on~$L$, and we let
\[
M_L^\bad(A) = \{ v\in M_L : \text{$A$ has bad reduction at $v$} \}.
\]
Then for~$P\in{A(L)}\setminus|\ASD|$, we define an
``inter\-sec\-tion-part'' and a ``Bernoulli-part'' of the global
canonical height via
\begin{align*}
\hhat_{A,\ASD}^\Bern(P) &= \frac{1}{[L:K]} \sum_{v\in M_L^\bad(A)} \lhat_{A,\ASD,v}^\Bern(P),\\
\hhat_{A,\ASD}^\Int(P) &= \frac{1}{[L:K]} \sum_{v\in M_L} \lhat_{A,\ASD,v}^\Int(P).
\end{align*}
With this
notation, there exists a \emph{local-global height constant}~$\kappa_{A,\ASD}$ so that
\begin{equation}
  \label{eqn:hPhPBhPIkapap}
  \hhat_{A,\ASD}(P) = \hhat_{A,\ASD}^\Bern(P) + \hhat_{A,\ASD}^\Int(P) - \kappa_{A,\ASD}
  \quad\text{for all~$P\in{A(\Kbar)}\setminus|\ASD|$.}
\end{equation}
We note that if~$\dim(A)=1$, i.e., if~$A$ is an elliptic curve, then
$\kappa_{A,\ASD}=0$.  However, if~$\dim(A)\ge2$,
then~$\kappa_{A,\ASD}$ is generally positive. 

Our main results are a Fourier series calculation and the following 
lower bound for the Bernoulli part of the canonical height in the case
that~$A$ is an abelian surface defined over a function field. This
theorem, and the Fourier averaging lemmas that we prove along the way,
are analogues of the key lemmas and results
in~\cite{hindrysilverman:lehmer}, where similar results are proven for
elliptic curves. However, we note that the main theorem
in~\cite{hindrysilverman:lehmer} is an unconditional Lehmer-type lower
bound for the canonical height on elliptic curves (with non-integral
$j$-invariant), while our result for abelian surfaces only gives a
lower bound for a suitable average of the Bernoulli part of the
height.

\begin{theorem}[Theorem~$\ref{theorem:hlen23len23}$ and
    Corollary~$\ref{corollary:conditionallehmer}$]    
  \label{theorem:mainthmintro}
Fix the following quantities\textup:
\begin{notation}
\item[$k$]
  an algebraically closed field of characterstic~$0$.
\item[$K/k$]
  a $1$-dimensional function field.
\item[$(A,\ThetaDivisor)/K$]
  an abelian variety~$A$ defined over~$K$ with an effective symmetric
  principal polarization $\Theta\in\Div_K(A)$.
\item[$\hhat_{A,\ThetaDivisor}$]
  the canonical height on~$A$ for the divisor $\ThetaDivisor$.
\item[$\hhat_{A,\ThetaDivisor}^\Bern,\hhat_{A,\ThetaDivisor}^\Int$]
  the Bernoulli and intersection parts of the canonical height on~$A$ for the divisor $\ThetaDivisor$.
\end{notation}
Assume that for every place~$v$ of~$K$, the abelian variety~$A$ has
either potential good reduction at~$v$ or totally multiplicative
reduction at~$v$, and that~$A$ has at least one place of
multiplicative reduction.
\begin{parts}
\Part{(a)}  
There are
constants~$\Cl[DZ]{jj10},\Cl[DZ]{jj11},\Cl[DZ]{jj8},\Cl[DZ]{jj9}>0$
and an integer~$d\ge1$ so that for all finite extensions~$L/K$ and all
sets of points $\Sigma\subset{A(L)}$
of~$\hhat_{A,\ThetaDivisor}$-height at most~$\Cr{jj10}$, there is a
subset~$\Sigma_0\subseteq\Sigma$ with~$\#\Sigma_0\ge\Cr{jj11}\#\Sigma$
so that the following double average\footnote{The averaging notation~$\Avg$  is
  fairly self-explanatory, but see
  Section~\ref{section:avgperiodicfuncs} for the precise definition.}
of the Bernoulli part of the heights of the points in~$\Sigma_0$
satisfies
\[
  \AvgL_{\substack{P,Q\in\Sigma_0\\ P\ne Q\\}}
  \AvgL_{T\in A[d]}
  \;\;
  \hhat_{A,\ThetaDivisor}^\Bern(P-Q+T)
  \ge
  \frac{\Cr{jj8}}{[L:K]^{2/3}} -  \frac{\Cr{jj9}}{\#\Sigma}.
\]
\Part{(b)}
Suppose that the subset~$\Sigma_0$ in~\textup{(a)} can always be
chosen 
so that it satisfies the further estimate
\begin{equation}
  \label{eqn:AvgPQAvgdIntgeLK23}
  \AvgL_{\substack{P,Q\in\Sigma_0\\ P\ne Q\\}}
  \AvgL_{T\in A[d]}
  \;\;
  \hhat_{A,\ThetaDivisor}^\Int(P-Q+T)
  \ge \kappa_{\ThetaDivisor},
\end{equation}
where~$\kappa_{\ThetaDivisor}$ is the constant appearing in~\eqref{eqn:hPhPBhPIkapap}.
Then there is a constant~$\Cl[DZ]{jj12}>0$ so that every non-torsion $P\in{A(\Kbar)}$ satisfies
\[
\hhat_{A,\ThetaDivisor}(P) \ge \frac{\Cr{jj12}}{\bigl[K(P):K\bigr]^2}.
\]
\end{parts}
\end{theorem}

We make several remarks, but again we refer the reader to
Section~\ref{section:surveyolderresults} for more details of the
history and known results surrounding Lehmer's conjecture.
To ease notation, we write
\[
D = \bigl[K(P):K\bigr]
\]
for the degree of the field of definition of~$P$ (although we note
that later we assign a different meaning to~$D$).

\begin{remark}
  Our proof uses the Fourier averaging technique that has previously
  been used for the classical Lehmer
  conjecture~\cite[Blanksby--Montgomery~(1971)]{MR0296021} and for
  Lehmer's conjecture on elliptic
  curves~\cite[Hindry--Silverman~(1990)]{hindrysilverman:lehmer}.  A
  crucial ingredient in the one-dimensional cases is that the Fourier
  series associated to the local heights has non-negative
  coefficients, a fact that is no long true in the higher dimensional
  case. Thus our proof has two key components. First we compute the
  relevant two-dimensional Fourier series attached to a periodic
  two-variable quadratic form with its associated hexagonal
  fundamental domain.  Second we
  deal with the issue that the Fourier series has both positive and
  negative Fourier coefficients via a subsidiary averaging process
  over a suitable collection of torsion points.
\end{remark}

\begin{remark}
  Currently the best known result for general abelian surfaces over number
  fields\footnote{Presumably Masser's proof carries over to the
    function field setting, where the~$\e$ might well be superfluous.}
  is due to Masser.  More generally it is proven in~\cite[Masser
    1984]{MR766295} that on an abelian variety of dimension~$g$, there
  is a Lehmer estimate
  \begin{equation}
  \label{eqn:masserhAPD2g62g}
  \hhat_A(P) \ge  \frac{C_\e(A/K)}{D^{2g+6+2/g+\e}}.
  \end{equation}
  Thus for abelian surfaces, i.e., for~$g=2$, Masser's lower bound
  is $O(D^{-11})$, which may be compared to our conditional lower bound
  of~$O(D^{-2})$ and with the conjectural lower bound
  of~$O(D^{-1/2})$.  Masser's proof uses auxiliary polynomials and
  methods from Diophantine approximation, a technique that has been
  long used in studying Lehmer's conjecture.
\end{remark}

\begin{remark}
  One would of course like to prove a result for number fields that is
  analogous to Theorem~\ref{theorem:mainthmintro}, much as was done
  in~\cite{hindrysilverman:lehmer} for elliptic curves.
  However, the
  Fourier expansion for the archimedean local height is likely to
  include negative Fourier coefficients, just as in the
  non-archimedean case.  And these negative Fourier coefficients would
  vitiate the argument used in the present paper, since we rely on the
  fact that our absolute values are discrete, and thus that the
  component groups on the N\'eron model are finite and are
  well-behaved under finite extension. For archimedean absolute
  values, the ``fiber'' on the ``N\'eron model'' should be viewed as
  having ``bad reduction'' with ``component group'' equal to a real
  torus. We thus have no consistent way to calculate which multiples
  of a point lie on (or near) the ``identity component'' in an
  archimedean topology.
\end{remark}

\begin{remark}
  Fourier averaging techniques have also been used successfully for
  studying Lang's height lower bound conjecture, in which one fixes a
  field~$K$ and varies the abelian variety. Lang's conjecture asserts
  (roughly) that for all abelian varieties of~$A/K$ of dimesion~$g$
  and all points~$P\in{A(K)}$ whose multiples are Zariski dense
  in~$A$, we have
  \[
  \hhat_A(P) \ge c_1(K,g) h(A/K) - c_2(K,g),
  \]
  where~$h(A/K)$ is an appropriate height of the abelian variety.
  For~$g=1$, this was proven for function fields, and conditionally
  for number fields assuming on Szpiro's conjecture, using Fourier
  averaging~\cite[Hindry--Silverman]{hindrysilverman:lehmer}. The
  difficulty in directly extending these proofs to abelian surfaces,
  even conditionally on a Szpiro-type conjecture, is again the
  two-fold problem of negative Fourier coefficients and that
  pesky~$\kappa_{A,\ASD}$ constant. However, see~\cite[David (1993)]{MR1254751}
  and~\cite[Pazuki
    (2013)]{MR3081000} for  Lang-style lower bounds for abelian
  varieties in which the lower bound has a correction term that
  measures the distance to the boundary of moduli space.
\end{remark}

\begin{remark}
  We also hope that it may be possible to extend our results to
  abelian varieties of dimension three or greater. However, it seems a
  challenging problem to write down an explicit formula for the
  Fourier series of the Bernoulli part of the local height in higher
  dimensions, since our two-dimensional hexagonal fundamental domain
  (see Figure~\ref{figure:hexagonandsquare}) would be replaced by
  a~$g$-dimensional parallelepiped. However, if this could be done, we
  would not be surprised if it could be used to prove a function field
  Lehmer-type bound with exponent~$2$ for the Bernoulli-part of the
  height.
\end{remark}

\section{Survey of Previous Results and Methods}
\label{section:surveyolderresults}
In this section we give a brief overview of the study of Lehmer-type
height lower bounds.  We continue with the notation
\[
D = \bigl[K(P):K\bigr].
\]
Lehmer's original conjecture~\cite[Lehmer~(1933)]{MR1503118}, actually
phrased as a question, says that\footnote{We assume throughout this
  historical survey (Section~\ref{section:surveyolderresults}) that
  ``trivial counter-examples'' are excluded. Thus for Lehmer's
  original conjecture, we assume that~$\a$ is non-zero and not a root of unity, for
  elliptic curves~$P$ is a non-torsion point, and for abelian
  varieties we assume that the iterates of~$P$ are Zariski dense.}
\[
h(\a) \ge CD^{-1}.
\]
General Lehmer-type estimates\footnote{We use the phrase ``Lehmer-type
  estimate'' to mean a height lower bound that decays at worst
  polynomially in the degree~$D$. We note that it is relatively easy
  to obtain exponentially decaying bounds.}  were proven in the
classical case in~\cite[Blanksby--Montgomery~1971]{MR0296021} using
Fourier series methods and in~\cite[Stewart~1978]{MR507748} using
auxiliary polynomials. Both proofs give bounds of the form
\[
h(\a) \ge CD^{-2}(\log D)^{-1}.
\]
Stewart's methods were applied
in~\cite[Dobrowolski~1979]{dobrowolski:lehmer}
to achieve the following bound, which is ever-so-close to Lehmer's
conjecture:
\[
h(\a) \ge CD^{-1} \left(\frac{\log\log D}{\log D}\right)^3.
\]
Dobrowolski's innovation was to use the Frobenius $p$-power map for
suitably many~$p$ to greatly increase the power of the vanishing
lemma.

The first general Lehmer-type estimate for elliptic curves was given
in~\cite[Anderson--Masser~(1980)]{MR591611}, where a lower
bound of roughly~$D^{-10}$ was given. This was subsequently improved
in~\cite[Masser~(1989)]{masser:lehmer} to
\[
\hhat_E(P) \ge C D^{-3}(\log D)^{-2}.
\]
Masser's proof uses auxiliary polynomials. A Fourier series proof of
the same precision was given
in~\cite[Zhang~(1989)]{zhang1989unpublished}.

Stronger results are known for restricted collections of elliptic
curves.  Notable is the result~\cite[Laurent (1983)]{laurent:lehmer},
who proves a Dobrowolski-type bound
\[
\hhat_E(P) \ge CD^{-1} \left(\frac{\log\log D}{\log D}\right)^3
\quad\text{if $E$ has complex multiplication.}
\]
And building on Zhang's ideas, Masser's result was improved
in~\cite[Hindry--Silverman~(1990)]{hindrysilverman:lehmer} to
\[
\hhat_E(P) \ge C D^{-2}(\log D)^{-2}
\quad\text{if $j(E)$ is non-integral.}
\]

There are also many results proving Lehmer-type estimates for points
defined over restricted types of fields. One of the earliest such
results is the proof~\cite[Smyth~(1971)]{smyth:lehmer} that Lehmer's
conjecture is true for all~$\a\in\Qbar^*$ such that~$\a^{-1}$ is not a
$\Qbar/\QQ$-Galois conjugate of~$\a$. (One says that~$\a$ is
non-reciprocal.)  Even stronger results are known for points defined
over abelian extensions of the ground field~$K$. It is shown
in~\cite[Amoroso--Dvornicich(2000)]{MR1740514} that
\[
h(\a) \ge C(K) > 0 \quad\text{for all non-zero non-roots of unity $\a\in K^\ab$,}
\]
and analogous estimates for points defined over~$K^\ab$ were proven
for elliptic curves in~\cite[Baker~(2003)]{MR1979685}
and~\cite[Silverman~(2004)]{MR2029512}, and then for abelian varieties
in~\cite[Baker--Silverman~(2004)]{MR2067482}. Note that in these
abelian extension results, the lower bounds are independent of~$D$.
Under the weaker assumption that~$K(P)/K$ is a Galois extension, it is
shown in~\cite[Galateau--Mah{\'e}~(2017)]{MR3598828} that the elliptic
Lehmer conjecture is true.

We next consider higher dimensional abelian varieties.  For a
simple abelian variety~$A/K$ of dimension~$g$ and appropriate choice
of canonical height, the current conjecture~\cite{MR1799933,MR766295}
appears to be
\[
\hhat_A(P) \ge C D^{-1/g},
\]
although no one has yet even managed to get even~$D^{-1}$. It is shown
in~\cite[Masser~(1984)]{MR766295} that
\[
\hhat_A(P) \ge C_\epsilon D^{-2g-6-2/g-\epsilon},
\]
and if~$A$ has complex multiplication, then Dobrowolski-type bounds
have been proven in~\cite[David--Hindry~(2000)]{MR1799933}
and~\cite[Ratazzi~(2008)]{MR2445828}.  For the $g$-fold product $E^g$ of an elliptic
curve, the estimate
\[
\hhat_{E^g}(P) \ge C D^{-1-1/2g}(\log D)^{-2/g}
\]
is proven in~\cite[Galateau--Mah{\'e}~(2017)]{MR3598828}.

A Lehmer-type conjecture involves fixing one geometric object such
as~$\GG_m$,~$E$, or~$A$ defined over a field~$K$, and finding height
lower bounds for points defined over extensions of~$K$. Dem'janenko
and Lang conjectured a different sort of height lower bound for
elliptic curves by fixing a field~$K$ and allowing the elliptic curve
to vary. The original conjecture had the form
\begin{multline*}
\hhat_E(P) \ge c_1(K)\log\Norm\Dcal_{E/K} - c_2(K) \\
\text{for all $E/K$ and all non-torsion $P\in E(K)$,}
\end{multline*}
and this has been generalized to abelian varieties with the
log-discri\-mi\-nant replaced by an appropriate height of the abelian
variety, e.g., the height~$h(A/K)$ used by Faltings in his proof of the
Mordell conjecture.  A Fourier series argument was used
in~\cite[Hindry--Silverman~(1988)]{hindrysilverman:integralpts} to
prove that Lang's conjecture for elliptic curves is a consequence of
Szpiro's conjectured inequality relating the discriminant and the
conductor of an elliptic curve, so in particular Lang's conjecture is
a theorem over one-dimensional characteristic~$0$ function fields.
However, for higher dimensional abelian varieties, the best known
estimates include an error term that grows as the moduli point of the
abelian variety approaches the boundary of the associated moduli
space; see for example~\cite[David~(1993)]{MR1254751}.


The definition of the canonical height of points on abelian varieties
can be extended to assign a canonical height to subvarieties of higher
dimension, and one can formulate a Lehmer conjecture and prove
Lehmer-type lower bounds for these higher dimesional heights. See for
example the series of papers
by David and Philippon~\cite{MR1478502,MR1949109,MR2355454}.

In this brief section, we have only touched on some of the work done
on Lehmer's conjecture.  For additional information, the reader might
consult the lengthy (unpublished) survey
article~\cite[Verger-Gaugry~(2019)]{VergerGaugrysurvey} that includes
an extensive bibliography of articles related to the conjectures of
Lehmer and Schinzel-Zassenhaus.

\section{An Overview of Canonical Local and Global Heights}
\label{section:overviewlocalhts}

We follow the exposition in Hindry's
notes~\cite{hindrynotesonlocalheights}; see also the
articles~\cite{MR1418354,MR1458753,MR1662481} by Werner
(especially~\cite{MR1458753}). We set the following
notation:\footnote{For comparison
  with~\cite{hindrynotesonlocalheights}, we note that Hindry's~$i_v(D,P)$
  is
  our~$\left\langle\overline{D}\cdot\overline{P}\right\rangle_{\Acal,v}$,
  and we have adopted his $B_{D,v}\bigl(j_v(P)\bigr)$ notation;
  cf.\ \cite[(3.10) and (3.11)]{hindrynotesonlocalheights}.  We also
  point the reader to the brief discussion of the function field
  setting in~\cite[Section~5]{hindrynotesonlocalheights}.}

\begin{notation}
\item[$k$]
  an algebraically closed field~$k$ of characteristic $0$.
\item[$K/k$]
  a $1$-dimensional function field over $k$.
\item[$M_L$]
  for finite extensions~$L/K$, 
  a complete set of absolute values on~$L$, normalized so that~$w(L^*)=\ZZ$
  for all~$w\in{M_L}$.
\item[$A/K$]
  an abelian variety of dimension~$g$ defined over~$K$.
\item[$\Acal$]
  the N\'eron model of $A/K$. 
\item[$\Acal^{\circ}$]
  the identity component of the N\'eron model of $A/K$.
\end{notation}

For~$P=[x_0,\ldots,x_N]\in\PP^N(\Kbar)$,
the \emph{Weil height} of~$P$ is defined by choosing a finite extension~$L/K$
with~$P\in\PP^N(L)$ and setting
\[
h(P)  = \frac{1}{[L:K]}\sum_{w\in M_L} \max\bigl\{-w(x_i)\bigr\}.
\]
The value is independent of the choice of~$L$.\footnote{Those who are
  familiar with the theory of Weil heights may wonder where the local
  factor~$[L_w:K_w]$ has gone.  The answer is that there is no residue
  degree, since our scalar field~$k$ is algebraically closed, and the
  ramification degree is already absorbed in the way that we have
  normalized the absolute values in~$M_K$ and~$M_L$, i.e.,
  if~$\a\in{K^*}$ and~$w\in{M_L}$ lies over~$v\in{M_K}$, then
  $w(\a)=e(w/v)v(\a)$ already includes the ramification degree.}

\begin{theorem}
\label{theorem:neronfncexist}
\textup{(N\'eron)}
Let~$A/K$ be an abelian variety. There exists is a unique
collection of functions
\[
\lhat_{\ASD,v} : A(\Kbar_v)\setminus|\ASD| \longrightarrow \RR,
\quad\text{where $\ASD\in\Div_K(A)$ and $v\in M_K$,}
\]
so that the following are true\textup:
\begin{parts}
\Part{(a)}
The map~$\lhat_{\ASD,v}$ is continuous, where we give~$A(\Kbar_v)$ the $v$-adic topology.
\Part{(b)}  
For all $\ASD,\ASD'\in\Div_K(A)$ and all $v\in{M_K}$,
\[
\lhat_{\ASD+\ASD',v} = \lhat_{\ASD,v} + \lhat_{\ASD',v}
\quad\text{on $A(\Kbar_v)\setminus\bigl(|\ASD|\cup|\ASD'|\bigr)$.}
\]
\Part{(c)}  
For all morphisms $\f:A\to{B}$ of abelian varieties over~$K$ and all $\ASD\in\Div_K(B)$,
\[
\lhat_{A,\f^*\ASD,v} = \lhat_{B,\ASD,v}\circ \f \quad\text{on $A(\Kbar_v)\setminus|\f^*\ASD|$.}
\]
\Part{(d)}
For all rational functions~$f\in{K(A)}$,
\[
\lhat_{\div(f),v} = v\circ f \quad\text{on $A(\Kbar_v)\setminus\bigl|\div(f)\bigr|$.}
\]
\Part{(e)}
\textup{(Normalization)}
For all $\ASD\in\Div_K(A)$ and all~$v\in{M_K}$, we have\footnote{Without
  this normalization, which N\'eron did not impose in his original
  formulation, the function~$\lhat_{\ASD,v}$ is only well-defined up to
  an~$M_K$-constant. We also mention that if the absolute value on~$K$
  is archimedean, then the normalization condition is equivalent to
  $\int_{A(\Kbar_v)} \lhat_{\ASD,v}(P)\,d\mu(P) = 0$, where~$\mu$ is Haar
  measure on~$A(\Kbar_v)\cong{A(\CC)}$.}
\begin{equation}
\label{eqn:limn2gPAn0}
  \lim_{N\to\infty} N^{-2g} \sum_{P\in A[N]\setminus|\ASD|} \lhat_{\ASD,v}(P)=0.
\end{equation}
\Part{(f)}
\textup{(Good Reduction)}
If~$A$ has potential good reduction at~$v$, then\footnote{N\'eron
  proved that this formula is true up to a
  constant. See~\cite{MR1413570} for a proof that the average of the
  intersection multiplicities over torsion points goes to~$0$, which
  implies that the constant vanishes.}
\[
\lhat_{\ASD,v}(P) =  \left\langle\overline{\ASD}\cdot\overline{P}\right\rangle_{\Acal,v}
\]
is the intersection index over~$v$ of the closures of~$\ASD$ and~$P$
in~$\Acal$. In the case of potential good reduction we have
\[
\lhat_{\ASD,v}(P) \ge 0 \quad\text{for all~$P\in A(\Kbar_v)\setminus|\ASD|$.}
\]
\Part{(g)}
\textup{(Bad Reduction)}
Let
\[
j_v : A(K) \longrightarrow (\Acal/\Acal^\circ)_v(k)
\]
be the homomorphism that sends a point to its image in the group of
components of the N\'eron model over~$v$. Then there is a
function\footnote{N\'eron further proved that the values
  of~$\BB_{\ASD,v}$ are rational numbers with denominators
  dividing~$2\#(\Acal/\Acal^\circ)_v(k)$.}
\[
\BB_{\ASD,v} : (\Acal/\Acal^\circ)_v(k) \longrightarrow \RR
\]
so that
\begin{equation}
  \label{eqn:LDvPivDPBBDvP}
  \lhat_{\ASD,v}(P) =  \left\langle\overline{\ASD}\cdot\overline{P}\right\rangle_{\Acal,v}
  + \BB_{\ASD,v}\bigl(j_v(P)\bigr) - \kappa_{\ASD,v},
\end{equation}
where again~$\kappa_{\ASD,v}$ is chosen so that~\eqref{eqn:limn2gPAn0} holds.
\Part{(h)}
\textup{(Local-Global Decomposition)} There is a constant
$\kappa_{\ASD}$ so that for all finite extensions~$L/K$ and all
$P\in{A(L)}\setminus|\ASD|$,\footnote{\textbf{Important Note}: When
  the local heights are normalized via~\eqref{eqn:limn2gPAn0}, then
  their weighted sum will generally differ from the glocal height by a
  non-zero constant that we have denoted~$\kappa_{A,\ASD}$;
  see~\cite[Appendix]{hindrynotesonlocalheights} for an
  example. However, if~$\dim(A)=1$, then~$\kappa_{A,\ASD}=0$, which is
  why this issue does not arise when working with elliptic curves.}
\[
\hhat_\ASD(P) = \frac{1}{[L:K]} \sum_{w\in M_K}  \lhat_{\ASD,w}(P)  - \kappa_{\ASD}.
\]
\end{parts}
\end{theorem}

\begin{definition}
\label{definition:intandbernlochts}
With notation as in Theorem~\ref{theorem:neronfncexist}, we define the
\emph{normalized intersection local height} and the \emph{normalized
  Bernoulli local height} to be, respectively,
\begin{equation}
  \label{eqn:deflambdaintandbern}
  \lhat_{\ASD,v}^\Int(P)
  =  \left\langle\overline{\ASD}\cdot\overline{P}\right\rangle_{\Acal,v} - \kappa_{\ASD,v}^\Int
  \quad\text{and}\quad
  \lhat_{\ASD,v}^\Bern(P)
  = \BB_{\ASD,v}\bigl(j_v(P)\bigr) - \kappa_{\ASD,v}^\Bern,
\end{equation}
where the constants~$\kappa_{\ASD,v}^\Int$ and~$\kappa_{\ASD,v}^\Bern$
are chosen to ensure the normalization formulas
\begin{equation}
  \label{eqn:normalizationIntBern}
   \frac{1}{N^{2g}}\sum_{T\in A[N]} \lhat_{\ASD,v}^\Int(T) \xrightarrow[N\to\infty]{} 0
  \quad\text{and}\quad
  \frac{1}{N^{2g}}\sum_{T\in A[N]} \lhat_{\ASD,v}^\Bern(T) \xrightarrow[N\to\infty]{} 0
\end{equation}
We note that Theorem~\ref{theorem:neronfncexist}(f) says that if~$A$
has potential good reduction at~$v$, then
\[
\lhat_{\ASD,v}=\lhat_{\ASD,v}^\Int,\quad
\kappa_{\ASD,v}^\Int=0,\quad\text{and}\quad
\lhat_{\ASD,v}^\Bern=0,
\]
so the added complication
of~\eqref{eqn:deflambdaintandbern}
and~\eqref{eqn:normalizationIntBern} are only needed if~$A$ does not
have potential good reduction.
\end{definition}

\begin{definition}
\label{definition:globalintbernhts}
We define the \emph{global intersection height}
and the \emph{global Bernoulli height} as follows: For~$P\in{A(\Kbar)\setminus|\ASD|}$, 
\begin{align*}
\hhat_{A,\ASD}^\Int(P) &= \frac{1}{\bigl[K(P):K\bigr]}  \sum_{w\in M_{K(P)}} \lhat_{A,\ASD,v}^\Int(P), \\*
\hhat_{A,\ASD}^\Bern(P) &= \frac{1}{\bigl[K(P):K\bigr]} \sum_{w\in M_{K(P)}}\lhat_{A,\ASD,v}^\Bern(P).
\end{align*}
\end{definition}

The next result summarizes how our various normalizations and
normalizing constants fit together.

\begin{proposition}
\label{proposition:hteqintplusbernhts}
With notation as in Theorem~$\ref{theorem:neronfncexist}$ and
Definitions~$\ref{definition:intandbernlochts}$
and~$\ref{definition:globalintbernhts}$, we have
\begin{align}
  \label{eqn:lhatexacteqlhatintpluslhatbern}
  \lhat_{\ASD,v}(P) &= \lhat_{\ASD,v}^\Int(P) + \lhat_{\ASD,v}^\Bern(P). \\
  \label{eqn:hhatsumintbernparts}
  \hhat_{\ASD}(P) &= \hhat_{\ASD}^\Int(P) + \hhat_{\ASD}^\Bern(P) - \kappa_{A,\ASD}.
\end{align}
\end{proposition}
\begin{proof}
Using~\eqref{eqn:limn2gPAn0},~\eqref{eqn:LDvPivDPBBDvP},
\eqref{eqn:deflambdaintandbern}, and \eqref{eqn:normalizationIntBern},
we see that
\begin{align*}
  0 &= 
  \lim_{N\to\infty} \frac{1}{N^{2g}}\sum_{T\in A[N]}
  \Bigl(\lhat_{\ASD,v}(T) - \lhat_{\ASD,v}^\Int(T) - \lhat_{\ASD,v}^\Bern(T)  \Bigr) \\
  &= \lim_{N\to\infty} \frac{1}{N^{2g}}\sum_{T\in A[N]}
  (-\kappa_{\ASD,v}^{\vphantom{\Int}}+\kappa_{\ASD,v}^\Int+\kappa_{\ASD,v}^\Bern) \\
  &=  -\kappa_{\ASD,v}^{\vphantom{\Int}}+\kappa_{\ASD,v}^\Int+\kappa_{\ASD,v}^\Bern.
\end{align*}
Thus~$\kappa_{\ASD,v}^{\vphantom{\Int}}=\kappa_{\ASD,v}^\Int+\kappa_{\ASD,v}^\Bern$, which
gives~\eqref{eqn:lhatexacteqlhatintpluslhatbern}. Then
\begin{align*}
  \hhat_\ASD(P) + \kappa_{A,\ASD}
  &= \frac{1}{[L:K]} \sum_{w\in M_K}   \lhat_{\ASD,w}(P)
  \quad\text{from Theorem~\ref{theorem:neronfncexist}(h),} \\
  &= \frac{1}{[L:K]} \sum_{w\in M_K}  
  \Bigl( \lhat_{\ASD,w}^\Int(P)+\lhat_{\ASD,w}^\Bern(P) \Bigr)
  \quad\text{from \eqref{eqn:lhatexacteqlhatintpluslhatbern},} \\
  &= \hhat_{\ASD}^\Int(P) + \hhat_{\ASD}^\Bern(P)
  \quad\text{from Definition~\ref{definition:globalintbernhts},}
\end{align*}
which proves~\eqref{eqn:hhatsumintbernparts}.
\end{proof}

\begin{remark}
\label{remark:lhatberndefeverywhere}
We note that~$j_v$ and N\'eron's Bernoulli function~$\BB_{\ASD,v}$ are
defined at all points, so the Bernoulli-part of the local height is
well-defined everywhere,
\[
\lhat_{\ASD,v}^\Bern : A(\Kbar) \longrightarrow \RR.
\]
This is in contrast to the intersection-part~$\lhat_{\ASD,v}^\Int$ of the
local height, which is only defined off of the support of the
divisor~$\ASD$, since if~$P\in|\ASD|$, then the local intersection
index~$\left\langle\overline{\ASD}\cdot\overline{P}\right\rangle_{\Acal,v}$
is not defined.
\end{remark}

\section{Local Heights for Completely Split Multiplicative Reduction}
\label{section:localhtformulanonarch}
We continue with our discussion of (local) heights based on the material
in~\cite{hindrynotesonlocalheights}.  For this section, we fix a
non-archimedean place~$v\in{M_K}$ such
that~$\Acal_v^\circ\cong\GG_m^g$ is a split torus. There is then a
$v$-adic uniformization
\[
\GG_m^g(K_v)/\Omega \xrightarrow{\;\;\sim\;\;} A(K_v),
\]
where~$\Omega$ is a (multiplicative) lattice, say spanned by the
columns of\footnote{The $q_{ij}$ may live in a multi-quadratic
  extension of~$K$.}
\[
\Omega
= \operatorname{Multiplicative-Span}
\left(
\begin{array}{c|c|c|c}
  q_{11}^2 & q_{12}^2 & \cdots & q_{1g}^2 \\
  q_{21}^2 & q_{22}^2 & \cdots & q_{2g}^2 \\
  \vdots & \vdots & \ddots & \vdots \\
  q_{g1}^2 & q_{g2}^2 & \cdots & q_{gg}^2 \\
\end{array}
\right).
\]
We define matrices
\[
\bfq = \begin{pmatrix}
  q_{11} & q_{12} & \cdots & q_{1g} \\
  q_{21} & q_{22} & \cdots & q_{2g} \\
  \vdots & \vdots & \ddots & \vdots \\
  q_{g1} & q_{g2} & \cdots & q_{gg} \\
\end{pmatrix}
\;\text{and}\;
Q = \begin{pmatrix}
  v(q_{11}) & v(q_{12}) & \cdots & v(q_{1g}) \\
  v(q_{21}) & v(q_{22}) & \cdots & v(q_{2g}) \\
  \vdots & \vdots & \ddots & \vdots \\
  v(q_{g1}) & v(q_{g2}) & \cdots & v(q_{gg}) \\
\end{pmatrix},
\]
where~$\bfq$ and~~$Q$ are symmetric, and~$Q$ is positive-definite.  In general,
when we apply~$v$ to vectors and matrices with entries in~$K_v$, we
mean the associated vector or matrix obtained by applying~$v$ to the
entries.  So for example, we have~$Q=v(\bfq)$, and
for~$\bfu\in\GG_m^g(K_v)$, we
have~$v(\bfu)=\bigl(v(u_1),\ldots,v(u_g)\bigr)$.

We introduce notation that will make it easier to work with linear
algebra on multiplicative spaces. For  (column)
vectors
\[
\bfu=(u_1,\ldots,u_g)\in\GG_m^g(K_v)
\quad\text{and}\quad
\bfm=(m_1,\ldots,m_g)\in\ZZ^m,
\]
we define\footnote{The intuition is that ${}^t\bfm\star\bfu$ is
  $\exp({}^t\bfm\log\bfu)$.}
\[
  {}^t\bfm\star\bfu = \prod_{i=1}^g u_i^{m_i}.
\]
Similarly, for the multiplicative period matrix~$\bfq$
and integer vectors $\bfm,\bfn\in\ZZ^g$, we define
\[
  {}^t\bfm\star \bfq\star \bfn = \prod_{i,j=1}^g q_{ij}^{m_in_j}.
\]
In particular, we note that
\[
v({}^t\bfm\star \bfq\star \bfn) = {}^t\bfm Q \bfn = \sum_{i,j=1}^g m_in_jv(q_{ij})
\]
is the value of the bilinear form associated to the positive-definite
matrix~$Q$.

Just as in the classical case over~$\CC$, the change-of-basis formula
for the multiplicative period matrix~$\bfq$ may be described using
the \emph{symplectic group}
\begin{multline*}
\Symp_{2g}(\ZZ) = \left\{
\begin{pmatrix} A&B\\ C&D\\ \end{pmatrix} \in \Mat_{2g\times2g}(\ZZ):
\right. \\
\left.
{\vrule height15pt depth0pt width0pt}^t\!\!
\begin{pmatrix} A&B\\ C&D\\ \end{pmatrix}
\begin{pmatrix} 0&I\\ -I&0\\ \end{pmatrix}
\begin{pmatrix} A&B\\ C&D\\ \end{pmatrix}
=
\begin{pmatrix} 0&I\\ -I&0\\ \end{pmatrix}
\right\}.
\end{multline*}
For our purposes, it suffices to describe the action of~$\Symp_{2g}(\ZZ)$
on the period valuation matrix~$Q$. It is given by the formula
\begin{equation}
  \label{eqn:symp2gaction}
  \begin{pmatrix} A&B\\ C&D\\ \end{pmatrix}\star Q = (AQ+B)(CQ+D)^{-1}.
\end{equation}

The following normalization lemma for the $2$-dimensional case will be
used later.

\begin{lemma}
\label{lemma:quadformwbpositive}
Let~$Q$ be a positive definite symmetric $2$-by-$2$ matrix.  Then
the~$\Symp_4(\ZZ)$ equivalence class of~$Q$ via the
action~\eqref{eqn:symp2gaction} contains a matrix
\[
\begin{pmatrix} a & b \\ b & c \\ \end{pmatrix} \in \Symp_4(\ZZ)\star Q
\]
satisfying
\begin{equation}
  \label{eqn:acgeb20le12blealec}
  ac > b^2 \quad\text{and}\quad 0\le 2b \le a \le c.
\end{equation}
We will say that a matrix~$\SmallMatrix{a&b\\b&c\\}$
satisfying~\eqref{eqn:acgeb20le12blealec} is \emph{normalized}.
\end{lemma}
\begin{proof}
Standard reduction theory of positive definite binary quadratic forms
(Gaussian reduction) tells us that there is a matrix~$A\in\SL_2(\ZZ)$
such that
\[
  A Q \,{}^t\!A = \begin{pmatrix} a & b \\ b & c \\ \end{pmatrix}
  \quad\text{with}\quad
  0\le |2b| \le a \le c.
\]
We note that
\[
A Q \,{}^t\!A
=
Q \star \begin{pmatrix} A & 0\\ 0 & {}^t\!A^{-1}\\ \end{pmatrix},
\quad\text{where}\quad
\begin{pmatrix} A & 0\\ 0 & {}^t\!A^{-1}\\ \end{pmatrix} \in \Symp_4(\ZZ).
\]
This completes the proof if~$b\ge0$. And if~$b<0$, then we can change
the sign of~$b$ using the following element of~$\Symp_4(\ZZ)$:
\[
\begin{pmatrix}
  -1 & 0 & 0 & 0 \\
  0 & 1 & 0 & 0 \\
  0 & 0 & -1 & 0 \\
  0 & 0 & 0 & 1\\
\end{pmatrix}
\star
\begin{pmatrix} a&b\\ b&c\\ \end{pmatrix}
=
\begin{pmatrix} a&-b\\ -b&c\\ \end{pmatrix}.
\]
This completes the proof of Lemma~\ref{lemma:quadformwbpositive}.
\end{proof}

\begin{definition}
The \emph{theta function} associated to the half-period matrix~$\bfq$
is the function
\begin{gather*}
\ThetaFunction(\,\cdot\,,\bfq) : \GG_m^g(K_v) \longrightarrow K_v,\\
\ThetaFunction(\bfu,\bfq) = \sum_{\bfm\in\ZZ^g} ({}^t\bfm\star\bfq\star\bfm)({}^t\bfm\star\bfu).
\end{gather*}
\end{definition}

Written out in full, we have
\[
\ThetaFunction(\bfu,\bfq) = \sum_{\bfm\in\ZZ^g}  \prod_{i,j=1}^g q_{ij}^{m_im_j} \cdot \prod_{k=1}^g u_k^{m_k}.
\]
The positive-definiteness of~$Q=v(\bfq)$ ensures that the sum converges
for all~$\bfu\in{\Kbar}_v^*$.

We next compute the transformation formula for~$\ThetaFunction$ when~$\bfu$ is
translated by an element of~$\Omega$. We observe that an element
of~$\Omega$ is a product of powers of the columns of the matrix whose
entries are~$q_{ij}^2$, so they are elements of~$\GG_m^g(K_v)$ of the
form~$\bfq\star2\bfn$ with~$\bfn\in\ZZ^g$.

\begin{proposition}
Let~$\bfn\in\ZZ^g$ and~$\bfu\in\GG_m^g(K_v)$.
\begin{parts}
\vspace{2\jot}
\Part{(a)}
$\displaystyle
\ThetaFunction\bigl(\bfu\cdot (\bfq\star2\bfn),\bfq\bigr)
=
({}^t\bfn\star\bfq\star\bfn)^{-1} ({}^t\bfn\star\bfu)^{-1} \ThetaFunction(\bfu,\bfq).
$
\vspace{2\jot}
\Part{(b)}
$\displaystyle
v\Bigl(\ThetaFunction\bigl(\bfu\cdot (\bfq\star2\bfn),\bfq\bigr)\Bigr)
=
v\Bigl(\ThetaFunction(\bfu,\bfq)\Bigr)
- {}^t\bfn{Q}\bfn-{}^t\bfn{v(\bfu)}.
$
\end{parts}
\end{proposition}
\begin{proof}
We give the elementary verification in
Appendix~\ref{section:verifyformulas}; see
Proposition~\ref{proposition:qprop1}.
\end{proof}


\begin{proposition}
The function
\begin{gather*}
\Lambda(\,\cdot\,,\bfq) : \GG_m^g(K_v) \longrightarrow \RR,\\
\Lambda(\bfu,\bfq) = v\bigl(\ThetaFunction(\bfu,\bfq)\bigr) + {\dfrac{1}{4}} {}^tv(\bfu) Q^{-1} v(\bfu),
\end{gather*}
is~$\Omega$-invariant, and hence descends to a function
\[
\Lambda(\,\cdot\,,\bfq) : A(K_v)\cong \GG_m^g(K_v)/\Omega \longrightarrow \RR.
\]
\end{proposition}
\begin{proof}
We give the elementary verification in
Appendix~\ref{section:verifyformulas}; see
Proposition~\ref{proposition:qprop2}.
\end{proof}

\begin{theorem}
\label{theorem:lDlDIlDBkk}
Let~$(A,\ThetaDivisor)/K_v$ be a principally polarized abelian surface
having totally split multiplicative reduction,
where~$\ThetaDivisor\in\Div_K(A)$ is an effective symmetric principal
polarization, and let~$\bfq\subset\GG_m^g(K_v)$ be an associated
multiplicative period matrix.
\begin{parts}
\Part{(a)}
There is a $2$-torsion point~$T_0\in{A[2]}$ so that
\[
\ThetaDivisor=\ThetaDivisorT+T_0
\quad\text{with}\quad
\ThetaDivisorT = \div\bigl(\ThetaFunction(\,\cdot\,,\bfq)\bigr). 
\]
\Part{(b)}
Let
\[
\Pcal : \GG_m^g(\Kbar_v) \longrightarrow A(\Kbar_v)
\]
denote the $v$-adic analytic uniformization of~$A$. Then
there is a $\kappa'_v\in\QQ$ so that for all $\bfu\in\GG_m^g(K_v)$,
\[
\lhat_{\ThetaDivisorT,v}\bigl(\Pcal(\bfu)\bigr) = v\bigl(\ThetaFunction(\bfu,\bfq)\bigr) +
{\dfrac{1}{4}} {}^tv(\bfu) Q^{-1} v(\bfu) - \kappa'_v.
\]
\Part{(c)}
Write
\begin{align*}
  \lhat_{\ThetaDivisorT,v}(P)
  &= \left\langle\overline{\ThetaDivisorT}\cdot\overline{P}\right\rangle_{\Acal,v}
  + \BB_{\ThetaDivisorT,v}\bigl(j_v(P)\bigr) - \kappa_v \\
  &= \lhat_{\ThetaDivisorT,v}^\Int(P) + \lhat_{\ThetaDivisorT,v}^\Bern(P)
\end{align*}
as in~\eqref{eqn:LDvPivDPBBDvP} and~\eqref{eqn:lhatexacteqlhatintpluslhatbern}. Then
\begin{align}
  \lhat_{\ThetaDivisor,v}^\Int(P+T_0) 
  &= \max_{\substack{\bfu\in\GG_m^g(\Kbar_v)\\ \Pcal(\bfu)=P\\}}
  v\bigl( \ThetaFunction(\bfu,\bfq) \bigr)  -  \kappa_v^\Int, \notag\\
  \lhat_{\ThetaDivisor,v}^\Bern(P+T_0) 
  &=\min_{\substack{\bfu\in\GG_m^g(\Kbar_v)\\ \Pcal(\bfu)=P\\}}
  {\dfrac{1}{4}} {}^tv(\bfu) Q^{-1} v(\bfu) - \kappa_v^\Bern.
  \label{eqn:lBThvPmin}
\end{align}
\end{parts}
\end{theorem}  

In the next section we are going to give an explicit formula for the
Fourier series of the Bernoulli local height~$\lhat_{\ThetaDivisor,v}^\Bern$
when~$\dim(A)=2$.  For notational reasons, it is easier to
renormalize the lattice and work with the standard
torus~$\RR^g/\ZZ^g$. Roughly speaking, we want to write
$\bfu\in\GG_m^g(\Kbar_v)$ as a (multiplicative) linear combination of
the lattice vectors.  But since we  only require the valuations, we define
a function
\begin{equation}
  \label{eqn:xuQinvvu}
  \bfx : \GG_m^g(\Kbar_v) \longrightarrow \RR^g, \quad
  \bfx(\bfu) = Q^{-1}v(\bfu).
\end{equation}
For~$P\in{A(\Kbar_v)}$, we write~$P=\Pcal(\bfu_P)$ for some choice
of~$\bfu_P\in\GG_m^g(\Kbar_v)$, and then we set
\[
\bfx_P = \bfx(\bfu_P) \in\RR^g.
\]
We note that~$\bfx_P$ is well defined in~$\RR^g/\ZZ^g$.

We associate to the period valuation matrix~$Q$ the ``periodic
quadratic form''
\begin{equation}
  \label{eqn:LQRgR}
  L_Q : \RR^g \longrightarrow \RR,\quad
  L_Q(\bfx_0) = \min_{\substack{\bfx\in\RR^g\\ \bfx\equiv\bfx_0\pmodintext{\ZZ^g}\\}} {}^t\bfx Q \bfx,
\end{equation}
and we write its associated Fourier series as
\[
L_Q(\bfx) = \sum_{\bfn\in\ZZ^g} \FC_Q(\bfn)e^{2\pi i {}^t\bfn\bfx}.
\]
Then
\begin{align}
  \label{equaton:LQhat0kappavB}
  {\dfrac{1}{4}} \FC_Q(\bfzero)
  &= \int_{\RR^g/\ZZ^g} {\dfrac{1}{4}} L_Q(\bfx)\,d\bfx \notag\\
  &= \lim_{\substack{N\to\infty\\\text{$N$ even}\\}} N^{-g} \sum_{\bft\in N^{-1}\ZZ^g/\ZZ^g} {\dfrac{1}{4}} L_Q(\bft) \notag\\
  &= \lim_{\substack{N\to\infty\\\text{$N$ even}\\}} N^{-2g} \sum_{T\in A[N]}
  \Bigl( \lhat_{\ThetaDivisor,v}^\Bern(T) + \kappa_v^\BB \Bigr)
  \quad\text{from \eqref{eqn:lBThvPmin},}\notag\\
  &= \kappa_v^\BB
  \quad\text{from \eqref{eqn:normalizationIntBern}.}
\end{align}
Then the Bernoulli local height is given by the formula
\begin{align*}
\lhat_{\ThetaDivisor,v}^\Bern(P+T_0)
&= \min_{\Pcal(\bfu)=P}   {\dfrac{1}{4}} {}^tv(\bfu) Q^{-1} v(\bfu) - \kappa_v^\BB
&&\text{from \eqref{eqn:lBThvPmin},} \\  
&= \min_{\Pcal(\bfu)=P}   {\dfrac{1}{4}} {}^t\bfx(u) Q \bfx(u) - \kappa_v^\BB
&&\text{from \eqref{eqn:xuQinvvu} and ${}^tQ=Q$,} \\
&= {\dfrac{1}{4}} L_Q(\bfx_P) - \kappa_v^\BB
&&\text{from \eqref{eqn:LQRgR},} \\
&= {\dfrac{1}{4}} L_Q(\bfx_P) - {\dfrac{1}{4}} \FC_Q(\bfzero)
&&\text{from \eqref{equaton:LQhat0kappavB}.}
\end{align*}

We record this result as a proposition.

\begin{proposition}
\label{proposition:lhatDvBLQxminusLhatQ}
With notation as in this and the previous section,
\[
\lhat_{\ThetaDivisor,v}^\Bern(P+T_0) = {\dfrac{1}{4}} L_Q(\bfx_P) - {\dfrac{1}{4}} \FC_Q(\bfzero)
\quad\text{for all $P\in A(\Kbar_v)$.}
\]
\end{proposition}

\begin{figure}[p]
\begin{picture}(300,300)(-150,-175)
  \thicklines
  \color{orange}   \polygon*(0,100)(100,100)(74.58,66.10)
  \color{black} \polygon(0,100)(100,100)(74.58,66.10)
  \color{red}   \polygon*(100,0)(100,100)(74.58,66.10)
  \color{black} \polygon(100,0)(100,100)(74.58,66.10)
  \color{blue}   \polygon*(0,-100)(-100,-100)(-66.10,-74.58)
  \color{black} \polygon(0,-100)(-100,-100)(-66.10,-74.58)
  \color{green}   \polygon*(-100,0)(-100,-100)(-66.10,-74.58)
  \color{black} \polygon(-100,0)(-100,-100)(-66.10,-74.58)
  \color{yellow}  \polygon*(-100,0)(-100,100)(0,100)(74.58,66.1)(100,0)(100,-100)(0,-100)(-66.10,-74.58)
  \color{black}
  \thinlines
  \put(0,0){\vector(1,0){120}}
  \put(0,0){\vector(-1,0){120}}
  \put(0,0){\vector(0,1){120}}
  \put(0,0){\vector(0,-1){120}}
  \thicklines
  \put(-100,100){\line(1,0){200}}
  \put(-100,-100){\line(1,0){200}}
  \put(100,-100){\line(0,1){200}}
  \put(-100,-100){\line(0,1){200}}
  \put(70,82){\makebox(0,0)[b]{\textbf{I}}}
  \put(90,68){\makebox(0,0)[t]{\textbf{II}}}
  \put(-65,-82){\makebox(0,0)[t]{\color{white}\textbf{III}}}
  \put(-85,-70){\makebox(0,0)[t]{\textbf{IV}}}
  \put(72,64){\makebox(0,0)[tr]{$\boldsymbol Q_{12}$}}
  \put(74.58,66.10){\circle*{5}}
  \put(-64,-72){\makebox(0,0)[bl]{$\boldsymbol Q_{34}$}}
  \put(-66.10,-74.58){\circle*{5}}  
  
  \put(150,-110){\makebox(0,0)[rt]{$\begin{aligned}
        Q_{12}  &= \left( \tfrac{c(a-b)}{2(ac-b^2)}, \tfrac{a(c-b)}{2(ac-b^2)}  \right) \\
        Q_{34} &= \left( -\tfrac{a(c-b)}{2(ac-b^2)}, -\tfrac{c(a-b)}{2(ac-b^2)}  \right) \\
      \end{aligned}$}}
\end{picture}
\begin{picture}(300,250)(-150,-150)
  \thicklines
  \color{orange}   \polygon*(0,-100)(100,-100)(74.58,-133.90)
  \color{black} \polygon(0,-100)(100,-100)(74.58,-133.90)
  \color{red}   \polygon*(-100,0)(-100,100)(-125.42,66.10)
  \color{black} \polygon(-100,0)(-100,100)(-125.42,66.10)
  \color{blue}   \polygon*(0,100)(-100,100)(-66.10,133.90)
  \color{black} \polygon(0,100)(-100,100)(-66.10,133.90)
  \color{green}   \polygon*(100,0)(100,-100)(125.42,-66.10)
  \color{black} \polygon(100,0)(100,-100)(125.42,-66.10)
  \color{yellow}  \polygon*(-100,0)(-100,100)(0,100)(74.58,66.1)(100,0)(100,-100)(0,-100)(-66.10,-74.58)
  \color{black}
  \thinlines
  \put(0,0){\vector(1,0){120}}
  \put(0,0){\vector(-1,0){120}}
  \put(0,0){\vector(0,1){120}}
  \put(0,0){\vector(0,-1){120}}
  \thicklines
  \put(-100,100){\line(1,0){200}}
  \put(-100,-100){\line(1,0){200}}
  \put(100,-100){\line(0,1){200}}
  \put(-100,-100){\line(0,1){200}}
\end{picture}
\caption{The hexagon where $F=L$, and the associated decomposition of
  the unit square as an octagon and four triangles}
\label{figure:hexagonandsquare}
\end{figure}
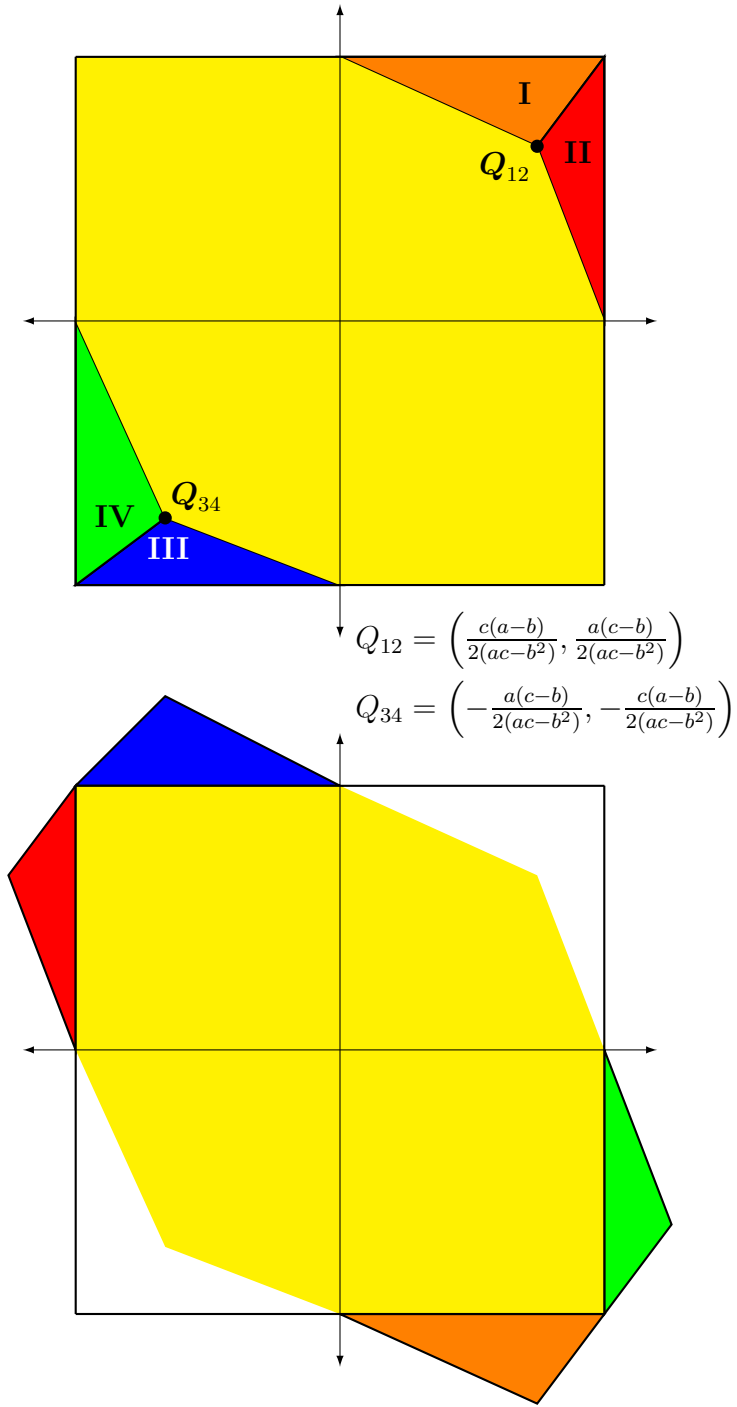

\section{A Hexagonal Fourier Calculation}
\label{section:hexfouriercalcquadform} 

\begin{definition}  
\label{definition:glossary}  
Figure~\ref{figure:notationFLabc}  gives a list of notation
and conventions that will remain fixed throughout the remainder of this article.
\end{definition}

\begin{figure}[ht]
\begin{center}
  \framebox{
  $
  \begin{aligned}
  \bfe(x) &= e^{2\pi i x},\quad \SIN(x)=\sin(2\pi x), \quad\text{and}\quad \COS(x)=\cos(2\pi x).\\
  a,b,c &\in\RR \quad \text{with}\quad a,c>0 \quad\text{and}\quad  D = ac-b^2 > 0. \\
  \a&=a-b \quad\text{and}\quad \g=c-b. \\
  F(x,y)
  &= ax^2 + 2bxy + cy^2\\
  &= \a x^2 + b(x+y)^2 + \g y^2.\\
  L(x,y)
  &= \min_{m,n\in\ZZ} F(x+m,y+n),\\
  \bfF_0 &:= \bfF_0(m,n) = c\a m+a\g n, \\
  \bfF_1 &:= \bfF_1(m,n) = c m - b n, \\
  \bfF_2 &:= \bfF_2(m,n) = a n - b m, \\
  \bfF_3 &:= \bfF_3(m,n) = \g m + \a n = \bfF_1 + \bfF_2.\\
  |L| &= |a,b,c| = \max\bigl\{|a|,|b|,|c|\bigr\}.
  \end{aligned}
  $
  }
\end{center}
\par\noindent
\parbox[][][l]{.95\hsize}{
If we need to specify~$a,b,c$ in the notation, we write
\begin{align*}
F(x,y) &= F_{a,b,c}(x,y) = F(a,b,c;x,y),\\
L(x,y) &= L_{a,b,c}(x,y) = L(a,b,c;x,y),
\end{align*}
and similarly for~$\bfF_0,\ldots,\bfF_3$.
We say that~$F$,~$L$, and~$(a,b,c)$ are \emph{normalized} if they
satisfy (cf.\  Lemma~\ref{lemma:quadformwbpositive})
\[
\framebox{$c\ge a \ge 2b \ge 0.$}
\]
\par
If we are working over the $v$-adic completion of a field, all of the
associated quantities will have a subscript~$v$, e.g.,~$(a_v,b_v,c_v)$
and $L_v$ and~$\bfF_{i,v}$.
}
\caption{Notation and Conventions and Formulas}
\label{figure:notationFLabc}  
\end{figure}

We note that the following formal identities are true in the
polynomial ring $\ZZ[a,b,c,m,n]$:
\begin{equation}
\label{eqn:idsforF0toF3}
\left.
\begin{array}{c}
\begin{aligned}
\bfF_0 - \a\bfF_1 & = Dn,\\
\bfF_0 - \g\bfF_2 & = Dm,\\
\bfF_0 + b\bfF_3 & = D(m+n),\\
\end{aligned}
\\[7\jot]
\begin{aligned}
  a \bfF_1 +  b \bfF_2 & = Dm,
  \hspace{1em}
  & b \bfF_1 + c \bfF_2 & = Dn,\\
  \a \bfF_1 + b \bfF_3 & = Dm,
  & -\g \bfF_1 + c \bfF_3 & = Dn,\\
  -\a \bfF_2 + a \bfF_3 & = Dm,
  & \g \bfF_2 + b \bfF_3 & = Dn.\\
\end{aligned}
\\
\end{array}
\right\}
\end{equation}

Our next result gives the Fourier expansion of~$L(x,y)$, which
is the~$\ZZ^2$-periodic version of the quadratic form~$F$.

\begin{theorem}
\label{theorem:fourierexpansionofL}
With notation as in
Figure~$\ref{figure:notationFLabc}$,
and in particular with~$L:(\RR/\ZZ)^2\to\RR$ being
the periodic function
\[
L(x,y) = \min_{\substack{\xi\in x+\ZZ\\ \eta\in y+\ZZ\\}} a\xi^2 + 2b\xi\eta+c\eta^2,
\]
the Fourier expansion
\[
L(x,y) = \sum_{m,n\in\ZZ} \FC(m,n)\bfe(mx+ny)
\]
of~$L(x,y)$ has Fourier coefficients given by the following
formulas:\footnote{We note that since~$ac\ne0$, the
  assumptions~$\bfF_1=0$ and~$(m,n)\ne(0,0)$ imply that~$n\ne0$, and
  similarly~$\bfF_2=0$ and~$(m,n)\ne(0,0)$ imply
  that~$m\ne0$. Further, the fact that~$D=ac-b^2\ne0$ tells us that at
  least one of~$\a=a-b$ and~$\g=c-b$ is non-zero, so~$\bfF_3=0$
  and~$(m,n)\ne(0,0)$ implies that at least one of~$m$ and~$n$ is
  non-zero. And if we normalize~$a,b,c$, then~$\a\g\ne0$, so
  $\bfF_3=0$ and~$(m,n)\ne(0,0)$ implies that both~$m$ and~$n$ are
  non-zero.}
\[
\FC(m,n) =
\begin{cases}
  \dfrac{a^2c+ac^2-2b^3}{12D} \quad \text{if $(m,n)=(0,0)$,}\hidewidth \\[3\jot]
  \dfrac{(-1)^n \a c^2}{2 \pi^2 D n^2}
  &\text{if $\bfF_1=cm-bn=0$, $(m,n)\ne(0,0)$,} \\[3\jot]
  \dfrac{(-1)^m \g a^2}{2 \pi^2 D m^2}
  &\text{if $\bfF_2=an-bm=0$, $(m,n)\ne(0,0)$,} \\[3\jot]
  \dfrac{(-1)^{m+n+1}  \a \g b}{2 \pi^2 D m n}
  &\text{if $\bfF_3=\g m+\a n=0$, $(m,n)\ne(0,0)$,} \\[3\jot]
  \dfrac{D^2 \displaystyle \SIN \left(\frac{ c \a  m  + a \g  n }{2D}\right)}
       {4 \pi ^3 (c m - b n)(a n - b m) ( \g  m + \a  n )  }
       \quad\text{otherwise.} \hidewidth\\
\end{cases}
\]
\end{theorem}
\begin{proof}
The region
\[
\Hcal = \bigl\{ (x,y)\in\RR^2 : F(x,y)=L(x,y) \bigr\}
\]
where~$F$ and~$L$ are equal is a hexagon, as shown in the bottom
illustration in Figure~\ref{figure:hexagonandsquare}. The intersection
of~$\Hcal$ with the unit square
\[
\Scal = \bigl\{ (x,y)\in\RR^2 : |x|\le\tfrac12,\;|y|\le\tfrac12 \bigr\}
\]
is the central octagon in both illustrations in
Figure~\ref{figure:hexagonandsquare}.  The set
difference~$\Scal\setminus\Hcal$ consists of four triangles, which are
shifted versions of the set difference~$\Hcal\setminus\Scal$ again as
shown in Figure~\ref{figure:hexagonandsquare}.  We label the four
triangular regions as follows:
\[
\begin{array}{l@{\quad}r@{}l} \hline
  \color{orange}
  \text{Region I} & L(x,y)&{}=F(x,y-1) \\ \hline
  \color{red}
  \text{Region II} & L(x,y)&{}=F(x-1,y) \\ \hline
  \color{blue}
  \text{Region III} & L(x,y)&{}=F(x,y+1) \\ \hline
  \color{green}
  \text{Region IV} & L(x,y)&{}=F(x+1,y) \\ \hline
\end{array}
\]

We start with some observations that we use in the computation of the
Fourier coefficients of~$L$.  The functions~$F$ and~$L$ have the
following symmetries:
\begin{align*}
  F_{a,b,c}( x,y) &= F_{a,b,c}( -x,-y) = F_{c,b,a}( y,x) , \\
  L_{a,b,c}( x,y) &= L_{a,b,c}( -x,-y) = L_{c,b,a}( y,x) .
\end{align*}
The sign change symmetry identifies Regions~I and~III, leading to the equality
\begin{multline*}
  \DI_{\textup{III}} \bigl\{ F_{a,b,c}(x,y-1) - F_{a,b,c}(x,y) \bigr\}\bfe(mx+ny) \\*
  = \DI_{\textup{I}} \bigl\{ F_{a,b,c}(x,y+1) - F_{a,b,c}(x,y) \bigr\}\bfe(-mx-ny),
\end{multline*}
and similarly for Regions~II and~IV. The reflection symmetry together
with the parameter swap $a\leftrightarrow{c}$ identifies Regions~I
and~II, leading to the equality
\begin{multline*}
  \DI_{\textup{II}} \bigl\{ F_{a,b,c}(x,y-1) - F_{a,b,c}(x,y) \bigr\}\bfe(mx+ny) \\*
  = \DI_{\textup{I}} \bigl\{ F_{c,b,a}(x,y-1) - F_{c,b,a}(x,y) \bigr\}\bfe(mx+ny),
\end{multline*}
and similarly for Regions~III and~IV. 
\par
Using these observations, we find that
\begin{align*}
  \FC & (m,n) = \DI_\Scal L(x,y)\bfe(mx+ny) \\
  &= \DI_\Scal F(x,y)\bfe(mx+ny) \\
  &\qquad {}+ \left( \DI_{\textup{I}} + \DI_{\textup{II}} + \DI_{\textup{III}} + \DI_{\textup{IV}} \right) \bigl\{ L(x,y)-F(x,y) \bigr\} \bfe(mx+ny) \\
  &= \DI_\Scal F(x,y)\bfe(mx+ny) \\
  & \qquad {}+ \DI_{\textup{I}} \bigl\{ F(x,y-1) - F(x,y) \bigr\}\bfe(mx+ny) \\
  & \qquad {}+ \DI_{\textup{II}} \bigl\{ F(x-1,y) - F(x,y) \bigr\}\bfe(mx+ny) \\
  & \qquad {}+ \DI_{\textup{III}} \bigl\{ F(x,y+1) - F(x,y) \bigr\}\bfe(mx+ny) \\
  & \qquad {}+ \DI_{\textup{IV}} \bigl\{ F(x+1,y) - F(x,y) \bigr\}\bfe(mx+ny) \\
  &= \DI_\Scal F(x,y)\COS(mx+ny) \\
  & \qquad {}+ 2\DI_{\textup{I}} \bigl\{ F(x,y-1) - F(x,y) \bigr\}\COS(mx+ny) \\
  & \qquad {}+ 2\DI_{\textup{II}} \bigl\{ F(x-1,y) - F(x,y) \bigr\}\COS(mx+ny) \\
  &= \int_{-\frac12}^{\frac12} \int_{-\frac12}^{\frac12} (ax^2+2bxy+cy^2)  \COS(mx+ny)\, dx\, dy \\
  & \qquad {}+ 2 \int_{\tfrac{a(c-b)}{2(ac-b^2)}}^{\tfrac12} \int_{\tfrac{c}{b}(\tfrac12-y)}^{\tfrac12-\tfrac{c-b}{a-b}(\tfrac12-y)} (c - 2 b x - 2 c y)\COS(mx+ny)
  \, dx \, dy \\
  & \qquad {}+ 2 \int_{\tfrac{c(a-b)}{2(ac-b^2)}}^{\tfrac12} \int_{\tfrac{a}{b}(\tfrac12-x)}^{\tfrac12-\tfrac{a-b}{c-b}(\tfrac12-x)} (a - 2 a x - 2 b y)\COS(mx+ny)
  \, dy \, dx.
\end{align*}  

It is now an easy task\footnote{Easy using a computer algebra system
  such as Mathematica, otherwise it is a feasible, but tedious, task.}
to compute these integrals. We start with the case that~$m$ and~$n$
are non-zero integers, and we assume for the moment that
\[
\bfF_1(m,n)\bfF_2(m,n)\bfF_3(m,n)\ne 0.
\]
Then the integral over Regions~I and~III is given explicitly by
\begin{align*}
  &\left( \DI_{\textup{I}} + \DI_{\textup{III}} \right) \bigl\{ L(x,y)-F(x,y) \bigr\} \bfe(mx+ny) \\
  &= 
\int_{\tfrac{a(c-b)}{2(ac-b^2)}}^{\tfrac12} \int_{\tfrac{c}{b}(\tfrac12-y)}^{\tfrac12-\tfrac{c-b}{a-b}(\tfrac12-y)}
\hspace*{-1.0em}
(c - 2 b x - 2 c y)\cdot \COS(mx+ny) \, dx \, dy \\
  &= \frac{b (a-b)}{4 \pi ^2 m ((c - b) m + (a - b) n)}  (-1)^{m+n} \\
  &\qquad{} +
  \frac{1}{4 \pi ^3 (c m - b n)}
  \cdot
  \left( \frac{a c - b^2}{(c - b) m + (a - b) n} \right)^2 \cdot \\
  &\hspace{10em} {}
  \SIN \left(\frac{ c (a - b) m  + a (c - b) n }{2(a c - b^2)}\right) \\
  &= \frac{(-1)^{m+n} b \a}{4 \pi ^2 m \bfF_3}   
  +
  \frac{D^2}{4 \pi ^3 \bfF_1 \bfF_3^2}
  \SIN \left(\frac{\bfF_0 }{2 D}\right).
\end{align*}
Further, as noted earlier, the integral over Regions~II and~IV is the
same with the swaps~$a\leftrightarrow{c}$
and~$m\leftrightarrow{n}$.\footnote{We note that these swaps
  leave~$\bfF_0$ and~$\bfF_3$ invariant, while swapping
  $\a\leftrightarrow\g$ and $\bfF_1\leftrightarrow\bfF_2$.}
Hence
\begin{align*}
  \left( \DI_{\textup{I}} \right. &
  \left. + \DI_{\textup{II}} + \DI_{\textup{III}} + \DI_{\textup{IV}} \right) \bigl\{ L(x,y)-F(x,y) \bigr\} \bfe(mx+ny) \,dx\,dy\\
  &=
  \left\{\frac{(-1)^{m+n} b \a}{4 \pi ^2 m \bfF_3} + \frac{D^2}{4 \pi ^3 \bfF_1 \bfF_3^2}  \SIN \left(\frac{\bfF_0 }{2 D}\right) \right\} \\
  &\omit\hfill$\displaystyle
  + \left\{\frac{(-1)^{m+n} b \g}{4 \pi ^2 n \bfF_3} + \frac{D^2}{4 \pi ^3 \bfF_2 \bfF_3^2}  \SIN \left(\frac{\bfF_0 }{2 D}\right) \right\}$\\
  &= \frac{(-1)^{m+n}b}{4\pi^2 m n}
  +  \frac{D^2{\displaystyle\SIN \left({ \bfF_0 }/{2D}\right)}} {4 \pi^3 \bfF_1 \bfF_2 \bfF_3}.
\end{align*}
On the other hand, the integral of~$F$ over the square is simply
\[
\DI_{\Scal} F(x,y)\COS(mx+ny) = \frac{(-1)^{m+n+1}b}{4\pi^2 m n},
\]
which cancels the first term in the sum of the four-triangle
integrals.\footnote{Presumably this cancellation is not a
    coincidence!} Hence
\begin{equation}
  \label{eqn:hatLmngen}
  \FC(m,n) =
  \frac{D^2{\displaystyle\SIN \left({ \bfF_0 }/{2D}\right)}} {4 \pi^3 \bfF_1 \bfF_2 \bfF_3}.
\end{equation}
One can check by a direct computation that the formula~\eqref{eqn:hatLmngen}
for $\FC(m,n)$ is valid if one, but not both, of~$m$ and~$n$
is~$0$. (This despite the fact that~$m$ and~$n$ seem to appear in the
denominators of some of the intermediate calculations.)

We note that we can use~\eqref{eqn:idsforF0toF3} to rewrite the
formula for~$\FC(m,n)$ so that the argument of the sine function is
instead related to one of the other~$\bfF_i$. Thus 
\begin{align*}
\SIN\left(\frac{\bfF_0}{2D}\right) &=
\SIN\left(\frac{\a\bfF_1+Dn}{2D}\right) =
(-1)^n \SIN\left(\frac{\a\bfF_1}{2D}\right),\\
\SIN\left(\frac{\bfF_0}{2D}\right) &=
\SIN\left(\frac{\g\bfF_2+Dm}{2D}\right) =
(-1)^m \SIN\left(\frac{\g\bfF_2}{2D}\right),\\
\SIN\left(\frac{\bfF_0}{2D}\right) &=
\SIN\left(\frac{-b\bfF_3+D(m+n)}{2D}\right) =
(-1)^{m+n+1} \SIN\left(\frac{b\bfF_3}{2D}\right).
\end{align*}
Substituting these into~\eqref{eqn:hatLmngen} gives three additional
formulas for~$\FC(m,n)$,
\begin{align}
  \label{eqn:hatLmngen1}
  \FC(m,n)
  &= \frac{(-1)^nD^2}{4\pi^3} \frac{\SIN(\a\bfF_1/2D\bigr)}{\bfF_1\bfF_2\bfF_3}  , \\
  \label{eqn:hatLmngen2}
  \FC(m,n)
  &= \frac{(-1)^mD^2}{4\pi^3} \frac{\SIN(\g\bfF_2/2D\bigr)}{\bfF_1\bfF_2\bfF_3} , \\
  \label{eqn:hatLmngen3}
  \FC(m,n)
  &= \frac{(-1)^{m+n+1}D^2}{4\pi^3} \frac{\SIN(b\bfF_3/2D\bigr)}{\bfF_1\bfF_2\bfF_3} .
\end{align}
\par
Continuing with our assumption that~$(m,n)\ne(0,0)$, we consider the
case that one of~$\bfF_1,\bfF_2,\bfF_3$ vanishes. An important
observation is that~\eqref{eqn:idsforF0toF3} and the fact that~$D\ne0$
tells us that at most one of these~$\bfF_i(m,n)$ can vanish.
\par
We fix an integer pair~$(m,n)\ne(0,0)$, we take a sequence of values
of~$(a,b,c)$ that cause one of the~$\bfF_i$ to vanish.  The integrals
that occur in the computation of~$\FC(a,b,c;m,n)$ are integrals of a
continuous function~$L$ over a compact set, where~$L$ depends
continuously on~$(a,b,c)$, so we can move the limit
as~$\bfF_i(a,b,c;m,n)\to0$ across the integral.
\par
Thus~\eqref{eqn:hatLmngen1} yields
\begin{align}
  \label{eqn:FCasF1to0}
  \lim_{\bfF_1\to0} \FC(m,n)
  &= \lim_{\bfF_1\to0} \frac{(-1)^nD^2}{4\pi^3} \frac{\SIN(\a\bfF_1/2D\bigr)}{\bfF_1\bfF_2\bfF_3} \notag\\
  &= \lim_{\bfF_1\to0} \frac{(-1)^nD^2}{4\pi^3\bfF_2\bfF_3}\cdot \frac{\sin(2\pi\a\bfF_1/2D\bigr)}{2\pi\a\bfF_1/2D}\cdot\frac{2\pi\a}{2D}\notag\\
  &=  \frac{(-1)^nD\a}{4\pi^2} \lim_{\bfF_1\to0} \frac{1}{\bfF_2\bfF_3} .
\end{align}
As $\bfF_1=cm-bn\to0$, we have (note that $c\ne0$)
\begin{align*}
\lim_{\bfF_1\to0} \bfF_2
&= \lim_{cm\to bn} an-bm
= \lim_{cm\to bn} an-\frac{b}{c}\cdot cm
= an-\frac{b}{c}\cdot bn
= \frac{Dn}{c}, \\
\lim_{\bfF_1\to0} \bfF_3
&= \lim_{cm\to bn} \g m + \a n
= \lim_{cm\to bn}  \frac{\g}{c}\cdot cm + \a n
= \frac{\g}{c}\cdot bn + \a n
= \frac{Dn}{c}.
\end{align*}
Substituting these two limits into~\eqref{eqn:FCasF1to0} yields
\[
  \lim_{\bfF_1\to0} \FC(m,n) = \frac{(-1)^n\a c^2}{4\pi^2 D n^2} .
\]
Similar computations using~\eqref{eqn:hatLmngen2}
and~\eqref{eqn:hatLmngen3} give the analogous formulas for the values
of~$\FC(m,n)$ as~$\bfF_2\to0$ and as~$\bfF_3\to0$. We leave the
details to the reader.
\par
It remains to compute~$\FC(0,0)$, for which the relevant
integrals are easy and left as an exercise.  This concludes the
proof of Theorem~\ref{theorem:fourierexpansionofL}.
\end{proof}

\section{Averaging (Periodic) Functions over (Torsion) Points}
\label{section:avgperiodicfuncs} 

We introduce a convenient notation for the expected value (average) of
a function over a set, and in particular over the $d$-torsion points
of an abelian group.

\begin{definition}
\label{definition:Avg}
Let~$S$ be a finite set, and let~$f:S\to\RR$ be a real-valued function. We write
\[
\AvgL_{x\in S} f(x) = \frac{1}{\#S} \sum_{x\in S} f(x).
\]
Similarly,
\[
\AvgL_{\substack{x,y\in S\\ x\ne y\\}} f(x-y) = \frac{1}{\#S^2-\#S}\sum_{\substack{x,y\in S\\ x\ne y\\}}f(x-y).
\]
If~$S=A$ is an abelian group and~$d\ge1$, by a slight abuse of notation we write
\[
(\Avg_d f)(x) = \AvgL_{t\in A[d]}f(x+t) =  \frac{1}{\#A[d]}\sum_{t\in A[d]} f(x+t)
\]
for the average of~$f$ at the $d$-torsion translates of~$x$, and we call~$\Avg_df$
the~\emph{$d$-average of~$f$}.
\end{definition}

\begin{example}
We illustrate Definition~\ref{definition:Avg} with three examples:
\begin{parts}
\Part{(1)}
For a function~$L:(\RR/\ZZ)^2\to\RR$ such as the one defined in
Theorem~\ref{theorem:fourierexpansionofL}, we have
\[
(\Avg_d L)(x,y) = \frac{1}{d^2} \sum_{i=0}^{d-1} \sum_{j=0}^{d-1} L\left(x+\frac{i}{d},y+\frac{j}{d}\right).
\]
\Part{(2)}
For an abelian variety~$A$ of dimension~$g$
and a function~$\lambda:A\to\RR$, we have
\[
(\Avg_d\lambda)(P) = \frac{1}{d^{2g}} \sum_{T\in A[d]} \lambda(P+T).
\]
\Part{(3)}
For any integer~$m$ and the function~$\bfe_m(x)=e^{2\pi i m x}$, we have
\[
(\Avg_d\bfe_m)(x) = \begin{cases}
\bfe_m(x) &\text{if $d\mid m$, } \\
0 &\text{if $d\nmid m$. } \\
\end{cases}
\]
\end{parts}
\end{example}

\begin{definition}
The \emph{2nd periodic Bernoulli polynomial} is the function defined by
\[
\text{$\BB_2(x) = x^2-x+\dfrac16$ for $0\le x\le 1$, and $\BB_2(x+n)=\BB_2(x)$ for $n\in\ZZ$.}
\]
The well-known Fourier expansion of~$\BB_2(x)$ is
\begin{equation}
  \label{eqn:B2fourier}
  \BB_2(x) = \frac{1}{2\pi^2} \sideset{}{^\prime}\sum_{k\in\ZZ} \frac{\bfe(kx)}{k^2},
\end{equation}
from which we  immediately obtain the distribution relation
\begin{equation}
  \label{eqn:B2distributionrelation}
  (\Avg_N\BB_2)(x) = \frac{1}{N^2}\BB_2(Nx).
\end{equation}
\end{definition}

We recall a Fej\'er kernel type estimate for~$\BB_2$.

\begin{lemma}
\label{lemma:avg2ndBernpoly}
Let~$R\ge1$ be an integer, and let
\[
T \subset \frac{1}{R}\ZZ \quad\text{with}\quad N=\#T
\]
be a set of~$N$ distinct rational numbers whose denominators
divide~$R$. Then
\[
\AvgL_{\substack{s,t\in T\\s\ne t\\}} \; \BB_2(s-t) \ge \frac{1}{6R^2} - \frac{1}{6(N-1)}.
\]
\end{lemma}
\begin{proof}
Let~$T=\{t_1,\ldots,t_N\}$. We compute
\begin{align*}
  \AvgL_{\substack{s,t\in T\\s\ne t\\}} {} & \; \BB_2(s-t) \\
  &= \frac{1}{N^2-N} \sum_{\substack{i,j=1\\i\ne j\\}}^N \BB_2(t_i-t_j) \\
  &= \frac{1}{N^2-N} \sum_{\substack{i,j=1\\i\ne j\\}}^N
  \frac{1}{2\pi^2} \sideset{}{^\prime}\sum_{k\in\ZZ} \frac{\bfe\bigl(k(t_i-t_j)\bigr)}{k^2}
  \quad\text{from \eqref{eqn:B2fourier},}  \\
  &= \frac{1}{2\pi^2(N^2-N)}
  \sideset{}{^\prime}\sum_{k\in\ZZ} \frac{1}{k^2}
  \sum_{\substack{i,j=1\\i\ne j\\}}^N  \bfe\bigl(k(t_i-t_j)\bigr) \\
  &= \frac{1}{2\pi^2(N^2-N)}
  \sideset{}{^\prime}\sum_{k\in\ZZ} \frac{1}{k^2}
  \biggl\{
  \underbrace{\left| \sum_{i=1}^N \bfe(k t_i) \right|^2}_{\hidewidth
    \substack{
      \text{this quantity is always $\ge0$,}  \\
      \text{and if $R\mid k$, then it equals $N^2$} \\
    }\hidewidth
  }
  - N \biggr\} \\
  &\ge \frac{1}{2\pi^2(N^2-N)} \sideset{}{^\prime}\sum_{k\in\ZZ} \frac{N^2-N}{(Rk)^2}
  - \frac{1}{2\pi^2(N^2-N)} \sideset{}{^\prime}\sum_{k\in\ZZ} \frac{N}{k^2} \\
  &= \frac{1}{2\pi^2 R^2} \cdot 2\z(2) - \frac{1}{2\pi^2 (N-1)} \cdot 2\z(2) \\
  &= \frac{1}{6R^2} - \frac{1}{6(N-1)}.
\end{align*}
This completes the proof of Lemma~\ref{lemma:avg2ndBernpoly}.
\end{proof}

We next express certain $d$-averages of the function~$L(x,y)$ in
Theorem~\ref{theorem:fourierexpansionofL} in terms of $d$-averages of
the second Bernoulli polynomial.

\begin{corollary} 
\label{corollary:avglambdaberntobern2}
Let~$a,b,c\in\ZZ$ with $D=ac-b^2>0$, let~$\a=a-b$ and~$\g=c-b$,
and let~$d$ by an integer satisfying
\begin{equation}
  \label{eqn:deqiuv02Dabc2q}
  d \equiv 0 \left(\bmod \frac{2D}{\gcd(a,b,c)^2} \right).
\end{equation}
Then the~$d$-average of the~$\ZZ^2$-periodic function
\[
L(x,y) = \min_{\substack{\xi\in x+\ZZ\\ \eta\in y+\ZZ\\}} a\xi^2 + 2b\xi\eta+c\eta^2
\]
is given by the formula
\begin{align}
\label{eqn:AvgdLxy3B2s}
\Avg_d L(x,y) = \FC(0,0)
&+
\frac{ \a (c,b)^2}{ D d^2} \BB_2\left( \frac{d (b x +  c y)}{\gcd(c,b)} \right) \notag\\*
&+
\frac{ \g (a,b)^2}{ D d^2} \BB_2\left( \frac{d (a x +  b y)}{\gcd(a,b)}  \right) \notag\\*
&+
\frac{ b(\a,\g)^2}{ D d^2} \BB_2\left( \frac{d (\a x -\g y)}{\gcd(\a,\g)} \right).
\end{align}
\end{corollary}
\begin{proof}
The congruence condition~\eqref{eqn:deqiuv02Dabc2q} says that~$d$ satisfies
\[
\frac{d\gcd(a,b,c)^2}{D}\in\ZZ,
\]
which in turn implies that
\[
\SIN\left( \frac{c\a d m + a \g d n}{2D} \right)
= \sin\left( \pi\cdot\frac{d\gcd(a,b,c)^2}{D}\cdot \frac{c\a m+a\g n}{\gcd(a,b,c)^2} \right) = 0,
\]
since~$a,c,\a,\g$ are all divisible by~$\gcd(a,b,c)$.  Then
Theorem~\ref{theorem:fourierexpansionofL} says that the associated
Fourier coefficient satisfies
\[
\FC(dm,dn) = 0 \quad\text{unless}\quad \bfF_1(m,n)\bfF_2(m,n)\bfF_3(m,n)=0.
\]
We note that if~$D\ne0$ and~$(m,n)\ne(0,0)$, then at most one
of the linear forms~$\bfF_1,\bfF_2,\bfF_3$ may vanish, so aside from~$\FC(0,0)$, the
Fourier series splits into three sums. We compute (note $d$ is even,
so the powers with~$(-1)^d$ may be omitted)
\begin{align*}
  \Avg_d L(x,y) - \FC(0,0)
  &=  \hspace{2em}
  \sideset{}{^\prime}\sum_{\hidewidth \substack{m,n\in\ZZ\\ \bfF_1(m,n)\bfF_2(m,n)\bfF_3(m,n)=0\\}\hidewidth }
  \FC(dm,dn) \bfe(dmx+dny) \\
&=
\sideset{}{^\prime}\sum_{\substack{m,n\in\ZZ\\ \bfF_1(m,n)=0\\}}
\dfrac{(-1)^{dn} \a c^2}{2 \pi^2 D (dn)^2} \bfe(dmx+dny) \\
&\qquad{}+
\sideset{}{^\prime}\sum_{\substack{m,n\in\ZZ\\ \bfF_2(m,n)=0\\}}
\dfrac{(-1)^{dm} \g a^2}{2 \pi^2 D (dm)^2} \bfe(dmx+dny) \\
&\qquad{}+
\sideset{}{^\prime}\sum_{\substack{m,n\in\ZZ\\ \bfF_3(m,n)=0\\}}
\dfrac{(-1)^{dm+dn+1}  \a \g b}{2 \pi^2 D (dm)(dn)} \bfe(dmx+dny)\\
&\omit\hfill \begin{tabular}[t]{l}
  using the formulas for $\FC(m,n)$\\ from Theorem~\ref{theorem:fourierexpansionofL},\\
  \end{tabular} \\
&=
\frac{ \a c^2}{2 \pi^2 D d^2}
\sideset{}{^\prime}\sum_{\substack{m,n\in\ZZ\\ \bfF_1(m,n)=0\\}}
\dfrac{1}{n^2} \bfe(dmx+dny) \\
&\qquad{}+
\frac{ \g a^2}{2 \pi^2 D d^2}
\sideset{}{^\prime}\sum_{\substack{m,n\in\ZZ\\ \bfF_2(m,n)=0\\}}
\dfrac{1}{m^2} \bfe(dmx+dny) \\
&\qquad{}+
\frac{ \a \g b}{2 \pi^2 D d^2}
\sideset{}{^\prime}\sum_{\substack{m,n\in\ZZ\\ \bfF_3(m,n)=0\\}}
\dfrac{-1}{mn} \bfe(dmx+dny).
\end{align*}
We rewrite the last three sums using
\begin{align*}
  \bigl\{ (m,n)\in\ZZ^2 : \bfF_1(m,n)=0 \bigr\} 
  &= \left\{ \left( \frac{bk}{(c,b)}, \frac{ck}{(c,b)} \right) : k\in\ZZ \right\}, \\
  \bigl\{ (m,n)\in\ZZ^2 : \bfF_2(m,n)=0 \bigr\} 
  &= \left\{ \left( \frac{ak}{(a,b)}, \frac{bk}{(a,b)} \right) : k\in\ZZ \right\}, \\
  \bigl\{ (m,n)\in\ZZ^2 : \bfF_3(m,n)=0 \bigr\} 
  &= \left\{ \left( \frac{\a k}{(\a,\g)}, \frac{-\g k}{(\a,\g)} \right) : k\in\ZZ \right\}.
\end{align*}
This yields
\begin{align}
  \label{Avgdfourierseriesz}
  \Avg_d L(x,y) - \FC(0,0)
  &=
  \frac{ \a (c,b)^2}{2 \pi^2 D d^2}
  \sideset{}{^\prime}\sum_{k\in\ZZ}
  \frac{1}{k^2} \bfe\left( \frac{d (b x +  c y)}{(c,b)}  k  \right) \notag\\
  &\qquad{}+
  \frac{ \g (a,b)^2}{2 \pi^2 D d^2}
  \sideset{}{^\prime}\sum_{k\in\ZZ}
  \frac{1}{k^2} \bfe\left( \frac{d (a x +  b y)}{(a,b)}  k  \right) \notag\\
  &\qquad{}+
  \frac{b(\a,\g)^2}{2 \pi^2 D d^2}
  \sideset{}{^\prime}\sum_{k\in\ZZ}
  \frac{1}{k^2} \bfe\left( \frac{d (\a x -\g y)}{(\a,\g)}  k  \right).
\end{align}
Using the Fourier series~\eqref{eqn:B2fourier} for~$\BB_2$ for the
three sums in~\eqref{Avgdfourierseriesz} gives the desired result.
\end{proof}

\section{Two Lower Bounds for the  Local Height}
\label{section:twolowerboundslocalheights}

In this section we prove two lower bounds for averages of the
Bernoulli part of the local height, one via Fourier averaging and one
via the pigeonhole principle. Both estimates will be used in the
proof of our Lehmer-type lower bound for the global height.  The
notation in Figure~\ref{figure:setupforlocallemmas} is used in the
statement of both lemmas.

\begin{figure}
\framebox{\begin{minipage}{0.95\linewidth} 
\begin{tabular}{cl}
  $K_v$ & \parbox[t]{0.75\linewidth}{
    a field that is complete with respect to a non-archimedean absolute value~$v$.
    } \\[5\jot]
$(A,\Theta)/K_v$ & \parbox[t]{0.75\linewidth}{
  an abelian variety~$A$ defined over~$K_v$ with an effective
  symmetric principal polarization $\Theta$, and such that~$A$ has
  totally split multiplicative reduction.
  } \\[9\jot]
$(a,b,c)$ & \parbox[t]{0.75\linewidth}{
  a normalized period valuation triple for $A/K_v$, i.e., if the period matrix is~$\bfq$,
  then \\
  \hspace*{2em} $a=v(q_{11}),\; b=v(q_{12})=v(q_{21}),\; c=v(q_{22}).$ 
  } \\[9\jot]
$D$ & ${}=ac-b^2$. \\
$d$ & \parbox[t]{0.65\linewidth}{
    a positive integer satisfying \\
     \hspace*{3em} $d \equiv 0 \left(\bmod \dfrac{2D}{\gcd(a,b,c)^2} \right).$
     } \\[10\jot]
$\Sigma$ &  a finite subset of $A(K_v)$. \\
  $N$ &  ${}=\#\Sigma$.
\end{tabular}
\end{minipage}
}
\caption{Notation and Setup for Lemmas \ref{lemma:fourieravgbound} and
  \ref{lemma:pigeonholebound}.}
\label{figure:setupforlocallemmas}
\end{figure}

\subsection{A Local Height Lower Bound via Fourier Averaging}
\label{section:localhtbdviafourier}

The main result of this section is an abelian surface analogue of the
elliptic curve result~\cite[Proposition~1.2]{hindrysilverman:lehmer}.
In order to handle the fact that for abelian surfaces, many of the
Fourier coefficients of the Bernoulli-part of the local height are
negative, the proof includes an average over~$d$-torsion points that
eliminates the negative coefficients. For our eventual application to
Lehmer-type height bounds, it is crucial that the value of~$d$ does
not change when the base field is replaced by a (ramified) extension.

\begin{lemma}
\label{lemma:fourieravgbound}
With notation as in Figure~$\ref{figure:setupforlocallemmas}$, we have
\footnote{We recall that although~$\lhat_{\ThetaDivisor,v}$
  is only defined on the complement of the support of its associated
  divisor, we can extend~$\lhat_{\ThetaDivisor,v}^\Bern$ to all
  of~$A(K_v)$. See Remark~\ref{remark:lhatberndefeverywhere}.}
\begin{multline*}
\smash[b]{ \AvgL_{\substack{P,Q\in\Sigma\\P\ne Q\\}} \AvgL_{T\in A[d]} }\; \lhat_{\ThetaDivisor,v}^\Bern(P-Q+T)  \\*
  \ge \frac{1}{24d^2} \left(
  \frac{\a+\g+b}{D} - \frac{\a(c,b)^2+\g(a,b)^2+b(\a,\g)^2}{D(N-1)} \right).
\end{multline*}
\end{lemma}
\begin{proof}
We first note that since we are averaging over~$d$-torsion points
and~$d$ is even, we may as well replace the principal
polarization~$\ThetaDivisor$ with the divisor of~$\theta(\bfu,\bfq)$,
since they differ by a $2$-torsion point that will disappear when we
take the average; cf.\ Theorem~\ref{theorem:lDlDIlDBkk}.
\par
An important observation is that for any point~$P$, the
vector~$(x_P,y_P)$ is given by the coordinates of~$P$ in the
group~$\ZZ^2/A\ZZ^2$ relative to the basis given by the columns of the
matrix~$A=\SmallMatrix{a&b\\b&c\\}$. Thus
\begin{multline}
  \label{eqn:xPyPabcduPvP}
  \begin{pmatrix}  x_P \\ y_P \\ \end{pmatrix}
  =
  \begin{pmatrix}  a&b\\ b&c\\ \end{pmatrix}^{-1}
  \begin{pmatrix}  u_P \\ v_P \\ \end{pmatrix}
  =
  \frac{1}{D}
  \begin{pmatrix}  c&-b\\ -b&a\\ \end{pmatrix}
  \begin{pmatrix}  u_P \\ v_P \\ \end{pmatrix} \\
  \quad\text{for some $u_p,v_P\in\ZZ$.}
\end{multline}
This yields the useful formulas
\begin{equation}
  \label{eqn:bxcyvaxbyu}
  bx_P+cy_P=v_P,\quad
  ax_P+by_P=u_p,\quad
  \a x_P-\g y_P=u_P-v_P.
\end{equation}
We also note that for any points~$P$ and~$Q$, we have
\begin{equation}
  \label{eqn:xPminusQxPminusxQ}
  x_{P-Q}\equiv x_P-x_Q\pmodintext{\ZZ} \quad\text{and}\quad y_{P-Q}\equiv y_P-y_Q\pmodintext{\ZZ}.
\end{equation}

To ease notation, we drop the~$\gcd$ from the notation~$\gcd(a,b)$.
We compute
\begin{align*}
  \AvgL_{\substack{P,Q\in\Sigma\\P\ne Q\\}} & \AvgL_{T\in A[d]}  \; 4 \lhat_{\ThetaDivisor,v}^\Bern(P-Q+T) \\*
  &= \smash[b]{ \AvgL_{\substack{P,Q\in\Sigma\\P\ne Q\\}} }
  \Avg_{d} \Bigl( L(x_{P-Q},y_{P-Q}) - \FC(0,0) \Bigr)  \\*
  &\omit\hfill\quad\text{from Proposition~\ref{proposition:lhatDvBLQxminusLhatQ},} \\
  &= \smash[b]{ \frac{1}{N^2-N}  \sum_{\substack{P,Q\in\Sigma\\P\ne Q\\}} }
  \biggl\{
  \frac{ \a (c,b)^2}{ D d^2} \BB_2\left( \frac{d (b x_{P-Q} +  c y_{P-Q})}{(c,b)} \right) \\*
  &\hspace{8em}
  +\frac{ \g (a,b)^2}{ D d^2} \BB_2\left( \frac{d (a x_{P-Q} +  b y_{P-Q})}{(a,b)}  \right) \\*
  &\hspace{8em}
  +\frac{ b(\a,\g)^2}{ D d^2} \BB_2\left( \frac{d (\a x_{P-Q} -\g y_{P-Q})}{(\a,\g)} \right)
  \biggr\} \\*
  &\omit\hfill from Corollary~\ref{corollary:avglambdaberntobern2}, \\
  &= \frac{ \a (c,b)^2}{ D d^2} \frac{1}{N^2-N} \sum_{\substack{P,Q\in\Sigma\\P\ne Q\\}}
  \BB_2\left( \frac{d (bx_P+cy_P) - d(bx_Q+cy_Q)}{(c,b)} \right) \\*
  &+\frac{ \g (a,b)^2}{ D d^2} \frac{1}{N^2-N} \sum_{\substack{P,Q\in\Sigma\\P\ne Q\\}}
  \BB_2\left( \frac{d(ax_P+by_P) - d(ax_Q+by_Q)}{(a,b)}  \right) \\*
  &+\frac{ b(\a,\g)^2}{ D d^2} \smash[b]{ \frac{1}{N^2-N} \sum_{\substack{P,Q\in\Sigma\\P\ne Q\\}} }
  \BB_2\left( \frac{d(\a x_P-\g y_P) - d(\a x_Q-\g y_Q)}{(\a,\g)} \right) \\*
  &\omit\hfill from \eqref{eqn:xPminusQxPminusxQ}, \\
  &= \frac{ \a (c,b)^2}{ D d^2} \frac{1}{N^2-N} \sum_{\substack{P,Q\in\Sigma\\P\ne Q\\}}
  \BB_2\left( \frac{d (v_P - v_Q)}{(c,b)} \right) \\*
  &+\frac{ \g (a,b)^2}{ D d^2} \frac{1}{N^2-N} \sum_{\substack{P,Q\in\Sigma\\P\ne Q\\}}
  \BB_2\left( \frac{d(u_P-u_Q)}{(a,b)}  \right) \\*
  &+\frac{ b(\a,\g)^2}{ D d^2} \frac{1}{N^2-N}
  \smash[b]{ \sum_{\substack{P,Q\in\Sigma\\P\ne Q\\}} }
  \BB_2\left( \frac{d\bigl( (u_P-v_P)-(u_Q-v_Q)\bigr)}{(\a,\g)} \right) \\*
  &\omit\hfill from \eqref{eqn:bxcyvaxbyu}, \\
  &\ge \frac{ \a (c,b)^2}{ D d^2}
    \cdot\frac{1}{6}\left( \frac{1}{(c,b)^2} - \frac{1}{N-1}  \right) \\*
  &+\frac{ \g (a,b)^2}{ D d^2}
  \cdot\frac{1}{6}\left( \frac{1}{(a,b)^2} - \frac{1}{N-1}  \right) \\*
  &+\frac{ b(\a,\g)^2}{ D d^2}
  \cdot\frac{1}{6}\left( \frac{1}{(\a,\g)^2} - \frac{1}{N-1}  \right) \\*
  &\omit\hfill from Lemma \ref{lemma:avg2ndBernpoly}, since $d,u_P,v_P\in\ZZ$.
\end{align*}
A little bit of algebra yields the desired result, which concludes the
proof of Lemma~\ref{lemma:fourieravgbound}.
\end{proof}

\subsection{A Local Height Lower Bound via the Pigeonhole Principle}
\label{section:pigeonholdbound}

The main result of this section is an analogue for abelian surfaces
of~\cite[Lemma~4]{MR747871}
and~\cite[Proposition~1.3]{hindrysilverman:lehmer}. However, the proof
is intrinsically more complicated than in the case of elliptic curves,
since it relies on a lower bound for the average of the local height over
a carefully chosen set of torsion points, and that lower bound
ultimately relies on the explicit Fourier expansion of the periodic
quadratic form given in Theorem~\ref{theorem:fourierexpansionofL}.

\begin{lemma}
\label{lemma:pigeonholebound}
With notation as in Figure~$\ref{figure:setupforlocallemmas}$, there
exists a subset~$\Sigma'\subseteq\Sigma$ containing
\[
\#\Sigma' \ge  6^{-3} \#\Sigma
\]
elements such that for all  distinct $P,Q\in\Sigma'$ we have
\[
  \Avg_d \lhat_{\ThetaDivisor,v}^\Bern(P-Q) \ge \frac{ \a (c,b)^2 + \g (a,b)^2 + b(\a,\g)^2}{144 D d^2}.
\]
\end{lemma}

\begin{proof}[Proof of Lemma $\ref{lemma:pigeonholebound}$]
As in the proof of Lemma~\ref{lemma:fourieravgbound}, the fact that
we're taking the~$d$-average with~$d$ even means that we may replace the
principal polarization~$\ThetaDivisor$ with the divisor
of~$\theta(\bfu,\bfq)$.
\par
We start with the formula
\begin{align}
  \label{eqn:AvgdLxy3B2st}
  \Avg_d 4 \lhat_{\ThetaDivisor,v}^\Bern(R)
  &= \Avg_d L(x_R,y_R) - \FC(0,0)
  \quad\text{from Proposition~\ref{proposition:lhatDvBLQxminusLhatQ},} \notag\\*
  &=
  \frac{ \a (c,b)^2}{ D d^2} \BB_2\left( \frac{d (b x_R +  c y_R)}{\gcd(c,b)} \right)
  \quad\text{from Corollary~\ref{corollary:avglambdaberntobern2},} \notag\\*
  &+
  \frac{ \g (a,b)^2}{ D d^2} \BB_2\left( \frac{d (a x_R +  b y_R)}{\gcd(a,b)}  \right) \notag\\*
  &+
  \frac{ b(\a,\g)^2}{ D d^2} \BB_2\left( \frac{d (\a x_R -\g y_R)}{\gcd(\a,\g)} \right).
\end{align}
\par
To ease notation, we momentarily define
\[
\|\,\cdot\,\|_\ZZ:\RR\longrightarrow \left[0,\frac12\right],\quad
\|t\|_\ZZ = \min_{n\in\ZZ} |t+n|,
\]
i.e.,~$\|t\|_\ZZ$ is the distance from~$t$ to the closest integer to~$t$.
It is easy to check that for all~$t\in\RR$, the periodic Bernoulli polynomial satisfies
\[
\|t\|_\ZZ \le\frac16
\quad\Longrightarrow\quad
\BB_2(t) \ge \frac{1}{36},
\]
since by periodicity and symmetry $\BB_2(-t)=\BB_2(t)$, it suffices to check for~$0\le{t}\le\frac16$.
Hence if~$R$ satisfies the three inequalities
\begin{equation}
  \label{eqn:3ineqforR}
  \left.
  \hspace*{4em}
  \begin{aligned}
  \left\|\dfrac{d (b x_R +  c y_R)}{(c,b)}\right\|_\ZZ&\le \dfrac16, \\
  \left\|\dfrac{d (a x_R +  b y_R)}{(a,b)}\right\|_\ZZ&\le \dfrac16, \\
  \left\|\dfrac{d (\a x_R -\g y_R)}{(\a,\g)}\right\|_\ZZ&\le \dfrac16,\\
  \end{aligned}
  \hspace*{4em}
  \right\}
\end{equation}
then each of the three Bernoulli polynomial values
appearing in~\eqref{eqn:AvgdLxy3B2st} is greater~$1/36$. This proves that
\[
\Avg_d \lhat_{\ThetaDivisor,v}^\Bern(R)
\ge
\frac{ \a (c,b)^2 + \g (a,b)^2 + b(\a,\g)^2}{36 D d^2}
\quad\text{if $R$ satisfies \eqref{eqn:3ineqforR}.}
\]
\par
We consider the map
\begin{align*}
\Sigma &\longrightarrow (\RR/\ZZ)^3,\\
R &\longmapsto
\left(\frac{d (b x_R +  c y_R)}{(c,b)},\frac{d (a x_R +  b y_R)}{(a,b)},\frac{d (\a x_R -\g y_R)}{(\a,\g)}\right).
\end{align*}
We divide the centered fundamental domain for~$(\RR/\ZZ)^3$ into~$6^3$
equally sized cubes whose sides have length~$6^{-1}$. Then the
pigeon-hole principle ensures that we can find a subset
\[
\Sigma'\subseteq\Sigma
\quad\text{with}\quad
\#\Sigma' \ge B^{-3}\#\Sigma
\]
such that the points in~$\Sigma'$ all lie in the same small cube.  It
follows that for all pairs~$P,Q\in\Sigma'$ we have
{\small
\begin{align}
\label{eqn:PQdif1}    
\left\|  \frac{d (b x_{P-Q} +  c y_{P-Q})}{(c,b)}  \right\|_\ZZ 
&=
\left\|  \frac{d (b x_P +  c y_P)}{(c,b)}  -  \frac{d (b x_Q +  c y_Q)}{(c,b)} \right\|_\ZZ  
\le \frac16,\\
\label{eqn:PQdif2}
\left\|  \frac{d (a x_{P-Q} +  b y_{P-Q})}{(a,b)}  \right\|_\ZZ 
&=
\left\|  \frac{d (a x_P +  b y_P)}{(a,b)}  -  \frac{d (a x_Q +  b y_Q)}{(a,b)} \right\|_\ZZ 
\le \frac16,\\
\label{eqn:PQdif3}
\left\|  \frac{d (\a x_{P-Q} +  \g y_{P-Q})}{(\a,\g)}  \right\|_\ZZ 
&=
\left\|  \frac{d (\a x_P +  \g y_P)}{(\a,\g)}  -  \frac{d (\a x_Q +  \g y_Q)}{(\a,\g)} \right\|_\ZZ 
\le \frac16.
\end{align}
}\ignorespaces
We note that for the three equalities
in~\eqref{eqn:PQdif1},~\eqref{eqn:PQdif2} and~\eqref{eqn:PQdif3}, we
are using the fact that the quantities~$x_{P-Q}$ and~$y_{P-Q}$ are
multiplied by integers. This combined with the fact that~$x_{P-Q}$
and~$y_{P-Q}$ satisfy
\[
x_{P-Q}\equiv x_P-x_Q\pmodintext{\ZZ} \quad\text{and}\quad y_{P-Q}\equiv y_P-y_Q\pmodintext{\ZZ}
\]
and the fact that we are using the norm on~$\RR/\ZZ$ justifies the
equalities.  Thus all differences of points in~$\Sigma'$
satisfy~\eqref{eqn:3ineqforR}, which completes the proof of
Lemma~\ref{lemma:pigeonholebound}.
\end{proof}

\section{A Bound for Small Differences Lying on $\Theta$}
\label{section:diffsontheta}

As noted earlier, the Bernoulli part of the local
height~$\lhat_{\ASD,v}^\Bern$ is defined at every point, but the
intersection part~$\lhat_{\ASD,v}^\Int$ is defined only away from the
support of the associated divisor~$\ASD$.  That means that if we want
to use the local-global decomposition of the global
height~$\hhat_\ASD$ described in
Theorem~\ref{theorem:neronfncexist}(h), we must restrict to points
lying in the complement~$A(\Kbar)\setminus|\ASD|$ of the support
of~$\ASD$. However, since ulimately we want to study points of small
height, it will suffice to use the following lemma, whose proof relies
on Ullmo and Zhang's proof of the Bogomolov conjecture.

\begin{lemma}
\label{lemma:diffsontheta}
Let~$\Kbar$ be an algebraically closed field of characteristic~$0$,
let~$A/\Kbar$ be an abelian surface, let~$\Theta\subset{A}$ be an
irreducible curve of genus at least~$2$, and let~$\hhat_A$ be a
canonical height on~$A$ relative to some ample symmetric divisor.
There are constants~$\Cl[DZ]{bg1},\Cl[DZ]{bg2}>0$ that
depend only on~$A/\Kbar$,~$\Theta$, and~$\hhat_A$ so that for all finite
subsets
\begin{equation}
  \label{eqn:SiginnhPlebg1}
  \Sigma\subset
  \Theta \cap 
  \bigl\{ P \in A(\Kbar) : \hhat_A(P) \le \Cr{bg1} \bigr\}
\end{equation}
there exists a subset~$\Sigma'\subset\Sigma$ satisfying
\[
\#\Sigma' \ge \Cr{bg2} \cdot \#\Sigma
\quad\text{and}\quad
(P-Q+A_\tors)\cap\Theta=\emptyset~\text{for all distinct $P,Q\in\Sigma'$.}
\]
\end{lemma}
\begin{proof}
The Bogomolov conjecture for  (curves on) abelian varieties, which was proven by
Ullmo~\cite{MR1609514} and Zhang~\cite{MR1609518}, says that there is a
constant~$\Cl[DZ]{bg3}>0$, depending only on~$A,\Theta,\hhat_A$, such
that the set
\[
\Xi = \Xi(A,\Theta,\hhat_A)
:= \Bigl( \Theta \cap \bigl\{ P\in A(\Kbar) : \hhat_A(P)\le\Cr{bg3} \bigr\}\Bigr) \quad\text{is finite.}
\]
In other words, there are a bounded number of points of~$A$ that lie
on~$\Theta$ and have small height.
\par
We set~$\Cr{bg1}=\frac14\Cr{bg3}$. Then
\begin{align*}
  P,Q\in\Sigma & \quad\text{and}\quad T\in A_\tors \quad\text{and}\quad   P-Q+T\in\Theta\\
  &\quad\Longrightarrow\quad
  \hhat_A(P-Q+T) = \hhat_A(P-Q) \le 2\hhat_A(P)+2\hhat_A(Q) \\
  &\omit\hfill parallelogram formula, \\
  &\quad\Longrightarrow\quad
  \hhat_A(P-Q+T) \le 4\,\Cr{bg1} \\
  &\omit\hfill from \eqref{eqn:SiginnhPlebg1}, since $P,Q\in\Sigma$, \\
  &\quad\Longrightarrow\quad
  \hhat_A(P-Q+T) \le \Cr{bg3},
  \quad\text{since $\Cr{bg1}=\frac14\Cr{bg3}$, } \\
  &\quad\Longrightarrow\quad
  P-Q+T \in \Xi.
\end{align*}
To ease notation, we let
\[
N = \#\Sigma \qquad\text{and}\qquad
\nu = \nu(A,\Theta,\hhat_A) := \max\bigl\{ \#\Xi, 2 \bigr\},
\]
and we let
\[
\Sigma = \{P_1,P_2,\ldots,P_N\}.
\]
\par
We build the set~$\Sigma'$ one step at a time. We first consider the
differences of~$P_1$ with the other elements of~$\Sigma$, translated by
torsion points, i.e., we consider the sets
\[
P_1-P_2+A_\tors,\;P_1-P_3+A_\tors,\;\ldots,\;P_1-P_N+A_\tors.
\]
The implication proven earlier implies that at most~$\nu=\#\Xi$ of
these sets may contain a point lying on~$\Theta$, so relabeling the elements
of~$\Sigma$, we have shown that
\begin{align*}
(P_1-P_2+A_\tors)\cap\Theta&=\emptyset,\\
(P_1-P_3+A_\tors)\cap\Theta&=\emptyset,\\
\omit\hfill$\vdots$\hfill\\
(P_1-P_{N-\nu}+A_\tors)\cap\Theta&=\emptyset.
\end{align*}
\par
We next consider the differences of~$P_2$ with the higher-indexed elements of~$\Sigma$,
again translated by torsion points,
\[
P_2-P_3+A_\tors,\; P_2-P_4+A_\tors,\;\ldots,\;P_2-P_{N-\nu}+A_\tors.
\]
As in the previous step, at most~$\nu$ of these sets contains
a point lying on~$\Theta$, so relabeling again, we have shown that
\[
(P_2-P_3+A_\tors)\cap\Theta=\emptyset,\;\ldots,\;
(P_2-P_{N-2\nu}+A_\tors)\cap\Theta=\emptyset.
\]
Continuing in this fashion, at the~$k$th step (until we run out of
points in~$\Sigma$), we will have shown that
\[
(P_k-P_{k+1}+A_\tors)\cap\Theta=\emptyset,\;\ldots,\;
(P_k-P_{N-k\nu}+A_\tors)\cap\Theta=\emptyset.
\]
This works as long as
\[
N-k\nu > k,\quad\text{and thus as long as}\quad k < \frac{N}{\nu+1}.
\]
Since~$\nu\ge2$ by assumption, we may certainly run the above
algorithm until~$k=\lceil{N/2\nu}\rceil$. Then by construction the set
\[
\Sigma' = \{P_1,P_2,\ldots,P_k\}
\]
has the property that
\[
(P_i-P_j+A_\tors) \cap \Theta = \emptyset \quad\text{for all $1\le i<j\le k$,}
\]
and the size of the set~$\Sigma'$ satisfies
\[
\#\Sigma' \ge \left\lceil\frac{N}{2\nu}\right\rceil \ge \frac{1}{2\nu}\#\Sigma.
\]
This completes the proof of Lemma~\ref{lemma:diffsontheta}
with~$\Cr{bg2}=1/2\nu$.
\end{proof}

\section{A Lehmer-Type Height Bound for Abelian Surfaces} 
\label{section:lehmerabeliansurface}

In this section we prove an unconditional, albeit somewhat technical,
lower bound for average values of the Bernoulli part of the canonical
height. We also prove a corollary giving an exponent~$2$ Lehmer-type
lower bound for the canonical height that is conditional on the
assumption that the average of the intersection part of the canonical
height is at least as large as the local-global
constant~$\kappa_\Theta$ appearing in
Theorem~\ref{theorem:neronfncexist}(h).

\begin{theorem}
\label{theorem:hlen23len23}
We set the following notation\textup:
\begin{notation}
\item[$k$]
  an algebraically closed field of characterstic~$0$.
\item[$K/k$]
  a $1$-dimensional function field.
\item[$(A,\ThetaDivisor)/K$]
  an abelian variety~$A$ defined over~$K$ with an irreducible effective symmetric
  principal polarization $\Theta\in\Div_K(A)$.
\item[$\hhat_{A,\ThetaDivisor}$]
  the canonical height on~$A$ for the divisor $\ThetaDivisor$.
\item[$\hhat_{A,\ThetaDivisor}^\Bern$]
  the Bernoulli part of the canonical height on~$A$ for the divisor $\ThetaDivisor$;
  see Definition~$\ref{definition:globalintbernhts}$.  
\end{notation}
Assume that for every place~$v$ of~$K$, the abelian variety~$A$ has
either potential good reduction at~$v$ or totally multiplicative
reduction at~$v$, and that~$A$ has at least one place of
multiplicative reduction.\footnote{For ease of exposition, we have
  excluded abelian surfaces having partial multiplicative reduction
  (surface with fibers $\Acal_v^\circ=\Ecal\rtimes\GG_m$ where~$\Ecal$
  is an elliptic curve), although we expect that these cases could be
  handled similarly. We also note that although the assumption
  that~$A$ have at least one place of potential multiplicative
  reduction is required for our proof, it is a relatively weak
  assumption. For example, if~$A/K$ has everywhere good reduction and
  is not isotrivial, then it necessarily has a non-simple
  fiber~$\Acal_v$, i.e., a fiber that is isogenous to a product of
  elliptic curves.} 
There are
constants~$\Cl[DZ]{jj1},\Cl[DZ]{jj2},\Cl[DZ]{jj3},\Cl[DZ]{jj4}>0$
and an integer~$d\ge1$ that depend only on~$A/K$ so that the following
holds\textup:
\par
For all finite extensions~$L/K$ and all sets of points
\begin{equation}
  \label{eqn:SigmainPALhPleC}
  \Sigma \subseteq \bigl\{ P \in A(L) : \hhat_{A,\ThetaDivisor}(P) \le \Cr{jj1} \bigr\},
\end{equation}
there is a subset $\Sigma_0\subseteq\Sigma$ having the following three properties\textup:
\begin{gather}
  \label{eqn:subset0geCsubset}
  \#\Sigma_0 \ge  \Cr{jj2}\cdot  \#\Sigma \\
  \label{eqn:PQTnotinThetadistPQTtors}
  P-Q+T \notin|\Theta|
  \quad\text{for all distinct $P,Q\in\Sigma_0$ and all $T\in A_\tors$.} \\
  \label{eqn:AvgPQAvgdBerngeLK23}
  \AvgL_{\substack{P,Q\in\Sigma_0\\ P\ne Q\\}}
  \AvgL_{T\in A[d]}
  \;\;
  \hhat_{A,\ThetaDivisor}^\Bern(P-Q+T)
  \ge
  \frac{\Cr{jj3}}{[L:K]^{2/3}} -  \frac{\Cr{jj4}}{\#\Sigma}.
\end{gather}
\end{theorem}

\begin{corollary}
\label{corollary:conditionallehmer}
With notation as in Theorem~$\ref{theorem:hlen23len23}$, suppose that
for every finite~$L/K$ and every set of points~$\Sigma$
satisfying~\eqref{eqn:SigmainPALhPleC}, there is a
subset~$\Sigma_0\subseteq\Sigma$
satisfying
\eqref{eqn:subset0geCsubset},~\eqref{eqn:PQTnotinThetadistPQTtors},~\eqref{eqn:AvgPQAvgdBerngeLK23},
and also\footnote{We note that~\eqref{eqn:PQTnotinThetadistPQTtors}
  ensures that~$\hhat_{A,\ThetaDivisor}^\Int$ is well-defined at all
  of the~$P-Q-T$ points under consideration.}
\begin{equation}
  \label{eqn:AvgPQAvgdIntgeLK23}
  \AvgL_{\substack{P,Q\in\Sigma_0\\ P\ne Q\\}}
  \AvgL_{T\in A[d]}
  \;\;
  \hhat_{A,\ThetaDivisor}^\Int(P-Q+T)
  \ge \kappa_{\ThetaDivisor},
\end{equation}
where~$\kappa_{\ThetaDivisor}$ is the constant appearing in
Theorem~\textup{\ref{theorem:neronfncexist}(h)}.
Then every non-torsion~$P\in{A(\Kbar)}$ satisfies
\[
\hhat_{A,\ThetaDivisor}(P) \ge \frac{\Cl[DZ]{jj5}}{\bigl[K(P):K\bigr]^2}.
\]
\end{corollary}

\begin{remark}
The assumption~\eqref{eqn:AvgPQAvgdIntgeLK23} in
Corollary~\ref{corollary:conditionallehmer} says roughly that (on
average) the intersection part of the local heights, by itself, is
sufficient to compensate for the difference between the canonical
height and the sum of the local heights. It is unclear to the authors
whether this is likely to be true, but we have included it in order to
explain how the somewhat technical estimate in
Theorem~\ref{theorem:hlen23len23} can be incorporated into the proof
of a Lehmer-type estimate, as was done unconditionally for elliptic
curves in~\cite{hindrysilverman:lehmer}.
\end{remark}

\begin{proof}[Proof of Theorem~$\ref{theorem:hlen23len23}$]
We first replace~$K$ by a finite extension over which~$A$ has
everywhere good or totally multiplicative reduction, which may require
some adjustment in the constants.
We let
\[
n = [L:K].
\]
As in the statement of the theorem, all of the constants may depend
on~$A/K$, but they are independent of~$L$,~$n$ and~$P\in{A(L)}$.  We
let
\[
S = \{v\in M_K : \text{$A$ has bad reduction at $v$} \}.
\]
For each~$v\in{S}$ we fix a uniformization
\[
\GG_m^2(\Kbar_v)\longrightarrow A(\Kbar_v)
\]
with kernel spanned (multiplicatively) by the columns of the matrix
\[
\bfq_v = 
\begin{pmatrix}
  q_{v,11} & q_{v,12} \\ q_{v,21} & q_{v,22} \\
\end{pmatrix}
\]
whose associated $\ThetaFunction$-function has divisor that is a
translation of~$\Theta$ be a $2$-torsion point.  The valuation matrix
\[
Q_v = v(\bfq_v) = 
\begin{pmatrix}
  v(q_{v,11}) & v(q_{v,12}) \\ v(q_{v,21}) & v(q_{v,22}) \\
\end{pmatrix}
=
\begin{pmatrix}
  a_v & b_v \\ b_v & c_v \\
\end{pmatrix}
\]
is symmetric and positive-definite. As usual, we let
\[
\a_v=a_v-b_v\quad\text{and}\quad \g_v=c_v-b_v.
\]
After a change of basis
as described in Lemma~\ref{lemma:quadformwbpositive}, we may assume that
the triple is normalized, and thus that
\[
D_v = a_vc_v-b_v^2 > 0
\quad\text{and}\quad
0\le 2b_v\le a_v\le c_v.
\]

To ease notation, we define two functions on~$\ZZ^3$, where
we note that the expressions~$\abcFunctionOne(a,b,c)$
and~$\abcFunctionThree(a,b,c)$ are the quantities appearing in both
Lemma~\ref{lemma:fourieravgbound} and
Lemma~\ref{lemma:pigeonholebound}:
\begin{align}
  \abcFunctionThree(a,b,c)
  &= \dfrac{D}{\gcd(a,b,c)^2}.
  \label{eqn:abcFunctionThree}  \\
  \abcFunctionOne(a,b,c) 
  &= \dfrac{\a\gcd(c,b)^2 + \g\gcd(a,b)^2 + b\gcd(\a,\g)^2}{D}.
  \label{eqn:abcFunctionOne}
\end{align}
For the proof of Theorem~\ref{theorem:hlen23len23}, it is crucial to
observe that these functions satisfy the homogeneity formulas
\[
\abcFunctionOne(ea,eb,ec)=e\abcFunctionOne(a,b,c)
\quad\text{and}\quad
\abcFunctionThree(ea,eb,ec)=\abcFunctionThree(a,b,c),
\]
since these homogeneity properties allow us to control the height
bounds as for ramified extensions~$L_w/K_v$.
\par
For~$w\in{M_L}$ with~$w\mid{v}$, we denote the ramification index
of~$w/v$ by~$e_w$, so~$w|_K=e_wv$. In particular, the valuations of
the multiplicative periods of~$A$ are multiplied by~$e_w$ when we move
from~$K$ to~$L$. Thus for places~$v$ of bad reduction, we have
\begin{equation}
  \label{eqn:aweqavetc}
  \left.
  \begin{aligned}
    a_w = e_w a_v,\quad b_w &= e_w b_v,\quad c_w = e_w c_v, \\
    \a_w = e_w \a_v,\quad \g_w &= e_w g_v, \\
    D_w = a_wc_w-b_w^2 &= e_w^2 D_v, \\
    \abcFunctionThree(a_w,b_w,c_w) &= \abcFunctionThree(a_v,b_v,c_v), \\
    \abcFunctionOne(a_w,b_w,c_w) &= e_w\abcFunctionOne(a_v,b_v,c_v). \\
   \end{aligned}
  \right\}
\end{equation}
We define the integer~$d$ by the formula
\[
d = 2 \LCM \bigl\{ \abcFunctionThree(a_v,b_v,c_v) : v \in S \bigr\}.
\]
We note that~$d$ depends only on~$A/K$, i.e., it is independent of the
extension field~$L/K$. We may thus replace~$L$ with the compositum
of~$L$ and~$K\bigl(A[d]\bigr)$, at the potential cost of
multiplying~$n=[L:K]$ be up to~$d^4$. Since~$d$ depends only on~$A/K$,
this requires only an adjustment of various constants. We henceforth
assume that
\[
A[d] \subset A(L).
\]
\par
We choose a place~$v_0\in{M_K}$ such that the fiber of the N\'eron
model of~$A$ is a torus, i.e., $\Acal_{v_0}(k)\cong\GG_m^2(k)$.
(By assumption, there is at least one such place.) Then
among the~$w\in{M_L}$ lying over~$v_0$, we choose~$w_0$ to have largest
ramfication index, i.e.,
\[
e_{w_0} = \max\{ e_w : w\in M_L,\,w\mid v \}.
\]
We also let
\[
M_{A/K}^\bad = \{v\in M_K : \text{$A$ has bad reduction at $v$} \},
\]
and similarly for~$M_{A/L}^\bad$.
\par

Let~$\Sigma$ be a set satisfying~\eqref{eqn:SigmainPALhPleC}.  We
start by applying Lemma~\ref{lemma:pigeonholebound}
to~$\Sigma\subset{A(L)}\subset{A(L_{w_0})}$ to find a
subset~$\Sigma'\subseteq\Sigma$ satisfying
\begin{equation}
  \label{eqn:NSig6n3S}
  \#\Sigma' \ge  6^{-3} \#\Sigma
\end{equation}
and such that for all  distinct $P,Q\in\Sigma'$ we have
\begin{equation}
  \label{eqn:avgdinlehpf}
  \Avg_d \lhat_{\ThetaDivisor,w_0}^\Bern(P-Q) \ge
  \frac{ \abcFunctionOne(a_{w_0},b_{w_0},c_{w_0})}{144 d^2}
  = \frac{ e_{w_0}\abcFunctionOne(a_{v_0},b_{v_0},c_{v_0})}{144 d^2}.
\end{equation}
\par
We next apply Lemma~\ref{lemma:diffsontheta} to the set~$\Sigma'$ to
find a subset~$\Sigma_0\subseteq\Sigma'$ satisfying\footnote{If we
  only want the lower bound on the Bernoulli part of the height, it is
  not necessary to use Lemma~\ref{lemma:diffsontheta}, since the
  Bernoulli part of the height is defined on all of~$A$. However, any
  application to the global height will need to also include the
  intersection part of the height, which is not defined on the support
  of~$\ThetaDivisor$.}
\begin{equation}
  \label{eqn:sigmaprimeprime}
  N := \#\Sigma_0 \ge \Cr{bg2} \cdot \#\Sigma'
\end{equation}
and
\begin{equation}
  \label{eqn:pminusqplustnotintheta}
  P-Q+T \notin|\Theta|
  \quad\text{for all distinct $P,Q\in\Sigma_0$ and all $T\in A_\tors$.}
\end{equation}
\par
We now estimate the double average~\eqref{eqn:AvgPQAvgdBerngeLK23} for
the set~$\Sigma_0$ and the integer~$d$.  We note
that~\eqref{eqn:pminusqplustnotintheta} ensures that the
points~$P-Q+T$ appearing in this calculation do not lie on the
divisor~$\Theta$, and thus the local heights are well-defined at all
such points.  Thus
\begin{align}
  \label{eqn:AvgAvglBern}
  \AvgL_{\substack{P,Q\in\Sigma_0\\ P\ne Q\\}}  \AvgL_{T\in A[d]} &
  \;\;  \hhat_{A,\ThetaDivisor}^\Bern(P-Q+T)  \notag\\
  &=
  \AvgL_{\substack{P,Q\in\Sigma_0\\ P\ne Q\\}} \AvgL_{T\in A[d]}
  \sum_{w\in M_{A/L}^\bad} \frac{1}{n} \l^\Bern_{\ThetaDivisor,w}(P-Q+T) \notag\\
  &=
  \sum_{w\in M_{A/L}^\bad} \frac{1}{n}
  \AvgL_{\substack{P,Q\in\Sigma_0\\ P\ne Q\\}} \AvgL_{T\in A[d]} \l^\Bern_{\ThetaDivisor,w}(P-Q+T).
\end{align}
We split the sum in~\eqref{eqn:AvgAvglBern} into three pieces:
\begin{parts}
  \Part{(1)}
  For the absolute value~$w_0$, we use the lower bound from
  Lemma~\ref{lemma:pigeonholebound}.
  \Part{(2)}
  For the absolute values~$w$ dividing~$v_0$ that are not equal
  to~$w_0$, we use the lower bound provided by the full strength of
  Lemma~\ref{lemma:fourieravgbound}.
  \Part{(3)}
  For the absolute values~$w$ with $w\in{M_{A/L}^\bad}$ that do not
  divide~$v_0$, we again use Lemma~\ref{lemma:fourieravgbound},
  but we discard the positive contribution coming from the~$1/D^2$
  terms.
\end{parts}
Carrying out these three estimates yields the following:
\begin{align}
  (1)\quad &
  \frac{1}{n}  \AvgL_{\substack{P,Q\in\Sigma_0\\ P\ne Q\\}} \AvgL_{T\in A[d]} \l^\Bern_{\ThetaDivisor,w_0}(P-Q+T) \notag\\
  &\quad{}\ge \frac{1}{n} \cdot
  \frac{ e_{w_0}\abcFunctionOne(a_{v_0},b_{v_0},c_{v_0})}{144 d^2}
  \quad\text{from~\eqref{eqn:avgdinlehpf},}\notag \\
  &\quad{}= \Cl[DZ]{dz1}\cdot \frac{e_{w_0}}{n}.
  \label{eqn:avgdestimate1}  \\
  (2)\enspace&
  \sum_{\substack{w\in M_{A/L}^\bad\\w\mid v_0,\, w\ne w_0\\}}
  \frac{1}{n}  \AvgL_{\substack{P,Q\in\Sigma_0\\ P\ne Q\\}} \AvgL_{T\in A[d]} \l^\Bern_{\ThetaDivisor,w}(P-Q+T) \notag\\
  &\quad{}\ge \smash[b]{ \sum_{\substack{w\in M_{A/L}^\bad\\w\mid v_0,\,w\ne w_0\\}} }
  \frac{1}{n}\cdot\frac{1}{24d^2} \left( \frac{\a_w+\g_w+b_w}{D_w}
  - \frac{\abcFunctionOne(a_w,b_w,c_w)}{N-1} \right) \notag \\
  &\omit\hfill applying Lemma~\ref{lemma:fourieravgbound} to $\Sigma_0$ and $w$, \notag \\
  &= \frac{1}{24 n d^2}  \smash[b]{ \sum_{\substack{w\in M_{A/L}^\bad\\w\mid v_0,\,w\ne w_0\\}} }
  \left( \frac{e_w(\a_{v_0}+\g_{v_0}+b_{v_0})}{e_w^2D_{v_0}}
  - \frac{e_w\abcFunctionOne(a_{v_0},b_{v_0},c_{v_0})}{N-1} \right) \notag \\
  &\omit\hfill using the homogeneity formulas~\eqref{eqn:aweqavetc}, \notag \\
  &= \frac{\a_{v_0}+\g_{v_0}+b_{v_0}}{24 n d^2 D_{v_0}}
  \sum_{\substack{w\in M_{A/L}^\bad\\w\mid v_0,\,w\ne w_0\\}} \frac{1}{e_w}
  - \frac{\abcFunctionOne(a_{v_0},b_{v_0},c_{v_0})}{24nd^2(N-1)}
  \sum_{\substack{w\in M_{A/L}^\bad\\w\mid v_0,\,w\ne w_0\\}} e_w \notag \\
  &= \frac{\a_{v_0}+\g_{v_0}+b_{v_0}}{24 n d^2 D_{v_0}}
  \biggl( \sum_{\substack{w\in M_{A/L}^\bad\\w\mid v_0,\,w\ne w_0\\}} \frac{1}{e_w} \biggr)
  - \frac{\abcFunctionOne(a_{v_0},b_{v_0},c_{v_0})(n-e_{w_0})}{24nd^2(N-1)}  \notag \\
  &\omit\hfill since in $\sum_{w\mid v}e_w=n$ for all $v$, \notag \\
  &\ge \frac{\Cl[DZ]{dz2}}{n} \sum_{\substack{w\in M_{A/L}^\bad\\w\mid v_0,\,w\ne w_0\\}} \frac{1}{e_w}
  - \frac{\Cl[DZ]{dz3}}{(N-1)}.
 \label{eqn:avgdestimate2}  \\
 (3)\enspace&
  \sum_{\substack{w\in M_{A/L}^\bad\\ w\nmid v_0\\}}
  \frac{1}{n}  \AvgL_{\substack{P,Q\in\Sigma_0\\ P\ne Q\\}} \AvgL_{T\in A[d]} \l^\Bern_{\ThetaDivisor,w}(P-Q+T) \notag\\
  &\ge\frac{1}{n}  \smash[b]{ \sum_{\substack{w\in M_{A/L}^\bad\\ w\nmid v_0\\}} }
  \frac{1}{24d^2} \left( - \frac{\abcFunctionOne(a_w,b_w,c_w)}{N-1} \right) \notag \\
  &\omit\hfill applying Lemma~\ref{lemma:fourieravgbound} to $\Sigma_0$ and $w$, \notag \\
  &= \frac{1}{n}  \smash[b]{ \sum_{\substack{w\in M_{A/L}^\bad\\ w\nmid v_0\\}} }
  \frac{1}{24d^2} \left( - \frac{e_w \abcFunctionOne(a_v,b_v,c_v)}{N-1} \right) \notag \\
  &\omit\hfill using the homogeneity formulas~\eqref{eqn:aweqavetc}, \notag \\
  &= - \frac{1}{24d^2n}   \sum_{\substack{v\in M_{A/K}^\bad\\ v\ne v_0\\}}
  \left(  \frac{\abcFunctionOne(a_v,b_v,c_v)}{N-1} \right)
  \sum_{\substack{w\in M_L\\ w\mid v\\}} e_w  \notag \\
  &= - \frac{1}{24d^2(N-1)}   \smash[b]{ \sum_{\substack{v\in S(A/K)\\ v\ne v_0\\}} }
  \abcFunctionOne(a_v,b_v,c_v)
  \notag \\
  &\omit\hfill since $\smash{\sum_{w\mid v}e_w=n}$, \notag \\
  &= \smash[t]{ - \frac{\Cl[DZ]{dz4}}{N-1}. }
 \label{eqn:avgdestimate3}  
\end{align}  

Substituting the sum of the three
estimates~\eqref{eqn:avgdestimate1},~\eqref{eqn:avgdestimate2},~\eqref{eqn:avgdestimate3}
into~\eqref{eqn:AvgAvglBern}, we find that
\begin{multline}
  \label{eqn:maxRSAKRCrdz34}
  \smash[b]{  \AvgL_{\substack{P,Q\in\Sigma_0\\ P\ne Q\\}}  \AvgL_{T\in A[d]}  }
  \;\;  \hhat_{\ThetaDivisor,\ThetaDivisor}^\Bern(P-Q+T)  \\
  \ge  \frac{1}{n} \biggl\{ 
  \Cr{dz1} e_{w_0}
  + \Cr{dz2} \sum_{\substack{w\in M_L\\w\mid v_0,\,w\ne w_0\\}} \frac{1}{e_w}
  \biggr\}
  - \frac{\Cr{dz3}+\Cr{dz4}}{N-1}.
\end{multline}
Since
\[
e_{w_0} = \max\{ e_w : w\mid v_0 \}
\quad\text{and}\quad
\sum_{w\mid{v_0}}e_w=n,
\]
we can apply Lemma~\ref{lemma:holderineq}
to the quantity in braces in~\eqref{eqn:maxRSAKRCrdz34} to obtain the
following lower bound, with newly relabeled constants depending
on~$A/K$ and where we have used~\eqref{eqn:NSig6n3S}
and~\eqref{eqn:sigmaprimeprime} to estimate $N=\#\Sigma_0$ in terms
of~$\#\Sigma$.\footnote{We remark that in order to apply
  Lemma~\ref{lemma:holderineq}, the integer~$n$ must satisfy
  $n^2\ge\Cr{dz2}/\Cr{dz1}$. There is no harm in our making this
  assumption, since these constants are given explicitly by
  \[
  \Cr{dz2} = \frac{\a_{v_0}+\g_{v_0}+b_{v_0}}{24d^2D_{v_0}} \quad\text{and}\quad
  \Cr{dz1} = \frac{\abcFunctionOne(a_{v_0},b_{v_0},c_{v_0}) }{144d^2}
  \ge \frac{\a_{v_0}+\g_{v_0}+b_{v_0}}{144d^2D_{v_0}},
  \]
  and thus $\Cr{dz2}/\Cr{dz1}\le6$. Hence it suffices to assume that~$n\ge3$.
}
\begin{align*}
  \AvgL_{\substack{P,Q\in\Sigma_0\\ P\ne Q\\}}  \AvgL_{T\in A[d]} 
  \;\;  \hhat_{\ThetaDivisor,\ThetaDivisor}^\Bern(P-Q+T)  
  &\ge  \frac{1}{n} \cdot \Cl[DZ]{dz5}n^{1/3}  - \frac{\Cl[DZ]{dz6}}{N-1} \\
  &\ge \frac{\Cr{jj3}}{n^{2/3}}   - \frac{\Cr{jj4}}{\#\Sigma}.  
\end{align*}
This completes the proof of Theorem~\ref{theorem:hlen23len23}.
\end{proof}

\begin{proof}[Proof of Corollary~$\ref{corollary:conditionallehmer}$]
Let~$P_0\in{A(\Kbar)}$ be a non-torsion point, and to ease notation, let
\[
L = K(P_0) \quad\text{and}\quad n = [L:K].
\]
We take~$M$ to be the largest integer satisfying
\begin{equation}
  \label{eqn:M2leChP0}
  M^2 \le \frac{\Cr{jj1}}{\hhat_{A,\Theta}(P_0)},
\end{equation}
where~$\Cr{jj1}$ is the constant appearing in~\eqref{eqn:SigmainPALhPleC}.
We consider the set of points
\[
\Sigma = \{ mP_0 : 0 \le m \le M-1 \}
\subset\bigl\{ P\in A(L) : \hhat_{A,\Theta}(P) \le \Cr{jj1} \bigr\},
\]
where the inclusion follows from~$\hhat_{A,D}(mP_0)=m^2\hhat_{A,D}(P_0)$
and our choice of~$M$.
\par
Then, according to~\eqref{eqn:subset0geCsubset},~\eqref{eqn:AvgPQAvgdBerngeLK23},
and~\eqref{eqn:AvgPQAvgdIntgeLK23}, we can find a subset~$\Sigma_0\subseteq\Sigma$
with $\#\Sigma_0\ge\Cr{jj2}\#\Sigma=\Cr{jj2}M$ that satisfies
\begin{align}
  \label{eqn:AvgPQAvgdBerngeLK23x}
  \AvgL_{\substack{P,Q\in\Sigma_0\\ P\ne Q\\}}
  \AvgL_{T\in A[d]}
  \;\;
  \hhat_{A,\ThetaDivisor}^\Bern(P-Q+T)
  &\ge
  \frac{\Cr{jj3}}{n^{2/3}} -  \frac{\Cr{jj4}}{M}. \\
  \label{eqn:AvgPQAvgdIntgeLK23x}
  \AvgL_{\substack{P,Q\in\Sigma_0\\ P\ne Q\\}}
  \AvgL_{T\in A[d]}
  \;\;
  \hhat_{A,\ThetaDivisor}^\Int(P-Q+T)
  &\ge \kappa_{\ThetaDivisor}.
\end{align}
Proposition~\ref{eqn:hhatsumintbernparts} says that
\[
\hhat_{A,\ThetaDivisor} = \hhat_{A,\ThetaDivisor}^\Int + \hhat_{A,\ThetaDivisor}^\Bern-\kappa_\ThetaDivisor,
\]
so adding~\eqref{eqn:AvgPQAvgdBerngeLK23x}
to~\eqref{eqn:AvgPQAvgdIntgeLK23x} yields
\begin{equation}
  \label{eqn:Avgjj3423}
  \AvgL_{\substack{P,Q\in\Sigma_0\\ P\ne Q\\}}
  \AvgL_{T\in A[d]}
  \;\;
  \hhat_{A,\ThetaDivisor}(P-Q+T)
  \ge
  \frac{\Cr{jj3}}{n^{2/3}} -  \frac{\Cr{jj4}}{M}.
\end{equation}
\par
But for any points~$P,Q\in\Sigma$ and for any torsion
point~$T\in{A_\tors}$, we have
\begin{align*}
\hhat_{A,\ThetaDivisor}(P-Q+T)
& = \hhat_{A,\ThetaDivisor}(P-Q) \\
& \le 2\hhat_{A,\ThetaDivisor}(P)+2\hhat_{A,\ThetaDivisor}(Q) \\
& \le 4 \max_{P\in\Sigma} \hhat_{A,\ThetaDivisor}(P) \\
& \le 4 \max_{0\le m < M} \hhat_{A,\ThetaDivisor}(mP_0) \\
& \le M^2 \hhat_{A,\ThetaDivisor}(P_0).
\end{align*}
Hence
\begin{equation}
  \label{eqn:AvgM2hP0}
  \AvgL_{\substack{P,Q\in\Sigma_0\\ P\ne Q\\}}
  \AvgL_{T\in A[d]}
  \;\;
  \hhat_{A,\ThetaDivisor}(P-Q+T)
  \le M^2 \hhat_{A,\ThetaDivisor}(P_0).
\end{equation}
Combining~\eqref{eqn:Avgjj3423} and~\eqref{eqn:AvgM2hP0} yields
\[
M^2 \hhat_{A,\ThetaDivisor}(P_0)
\ge   \frac{\Cr{jj3}}{n^{2/3}} -  \frac{\Cr{jj4}}{M}.
\]
Setting~$M$ to be the smallest integer satisfying
\begin{equation}
  \label{eqn:Mge2Cn23}
  M \ge \frac{2\Cr{jj4}n^{2/3}}{\Cr{jj3}}
\end{equation}
yields (after adjusting constants)
\[
n^{4/3} \hhat_{A,\ThetaDivisor}(P_0) \ge \frac{\Cr{jj5}}{n^{2/3}}.
\]
\par
This completes the proof of
Corollary~\ref{corollary:conditionallehmer} provided that we can
justify choosing~$M$ to satisfy~\eqref{eqn:Mge2Cn23}, since we earlier
in~\eqref{eqn:M2leChP0} assumed that~$M$ satisfies an upper bound.
In other words, we need to check that there is an integer~$M$ in the interval
\[
\frac{2\Cr{jj4}n^{2/3}}{\Cr{jj3}} \le M \le \sqrt{\frac{\Cr{jj1}}{\hhat_{A,\ThetaDivisor}(P_0)}}.
\]
But if there is no such~$M$, then we find that
\[
\sqrt{\frac{\Cr{jj1}}{\hhat_{A,\ThetaDivisor}(P_0)}}
\le \frac{2\Cr{jj4}n^{2/3}}{\Cr{jj3}} + 1,
\]
and squaring both sides and adjusting constants, we see that
\[
\hhat_{A,\ThetaDivisor}(P_0) \ge \frac{\Cl[DZ]{jj6}}{n^{4/3}},
\]
which is an even stronger inequality than the one that we are trying
to prove.
\end{proof}

The following is a more precise and fully explicated version
of~\cite[Lemma~3.1]{hindrysilverman:lehmer}.

\begin{lemma}
\label{lemma:holderineq}
Let $\a,\b,n>0$ be positive real numbers satisfying
\begin{equation}
  \label{eqn:n2gebetaalpha}
  n^2 \ge \b/\a,
\end{equation}
and let $e_0,\ldots,e_r>0$  be positive real numbers satisfying
\[
e_0 = \max\{e_0,\ldots,e_r\}
\quad\text{and}\quad
n = e_0+\cdots+e_r.
\]
Then
\begin{equation}
  \label{eqn:alphae0betaei}
  \a e_0 + \b \sum_{i=1}^r \frac{1}{e_i} \ge (\a^2 \b n)^{\frac{1}{3}}.
\end{equation}
\end{lemma}
\begin{proof}
Since~$e_0$ is the largest of the~$e_i$ and~$n$ is
the sum of the~$e_i$, we can estimate
\begin{equation}
  \label{eqn:e0genr1}
  e_0 \ge \frac{e_0+\cdots+e_r}{r+1} = \frac{n}{r+1}.
\end{equation}
We compute
\begin{align}
\label{eqn:r2ge4nr1ei1}
r^2&= \left( \sum_{i=1}^r e_i^{1/2}\cdot e_i^{-1/2} \right)^2 \notag \\
&\le \left( \sum_{i=1}^r e_i \right)\left( \sum_{i=1}^r e_i^{-1} \right)
&&\text{Cauchy-Schwartz inequality,} \notag\\
&= (n-e_0)\left( \sum_{i=1}^r e_i^{-1} \right)
&&\text{since $e_0+\cdots+e_r=n$,} \notag\\
&\le \frac{rn}{r+1} \left( \sum_{i=1}^r e_i^{-1} \right)
&&\text{using \eqref{eqn:e0genr1}.}
\end{align}
We use this estimate to bound the left-hand side
of~\eqref{eqn:alphae0betaei} as
\begin{align}
\label{eqn:ae0bsumi1rinf}
\a e_0 + \b \sum_{i=1}^r \frac{1}{e_i}
&\ge \a \frac{n}{r+1} + \b \frac{r^2+r}{n}
\quad\text{using \eqref{eqn:e0genr1}  and \eqref{eqn:r2ge4nr1ei1},} \notag\\
&\ge \inf_{t>0} \left\{ \frac{\a n}{t+1} + \frac{\b}{n}(t^2+t) \right\} \notag\\
&= \inf_{x>1} \left\{ \frac{\a n}{x} + \frac{\b}{n}(x^2-x) \right\}
\quad\text{setting $x=t+1$,} \notag\\
&= (\a^2\b n)^{1/3} \inf_{u>\g} \left\{ \frac{1}{u}+ u^2 - \g u \right\} \\
&\omit\hfill setting $\g=\left(\dfrac{\b}{\a n^2}\right)^{1/3}$ \hspace*{-10pt} and $u=\g x$. \notag
\end{align}

To ease notation, we let
\[
f(\g,u) = u^{-1} + u^2 - \g u.
\]
The fact that
\[
\frac{d^2\phantom u}{du^2}(u^{-1}+u^2-\g u) = 2u^{-3} + 2 > 0 \quad\text{for all $u>0$}
\]
shows that~$f(\g,u)$ has at most one minimum on the half-line~$u>0$,
and then the fact that~$f(\g,u)\to\infty$ as~$u\to0^+$ and
as~$u\to\infty$ shows that it has a unique minimum. We thus get a
well-defined function
\[
F(w) = \inf_{u>0} f(w,u) = \inf_{u>0} \{ u^{-1} + u^2 - wu \} \quad\text{for $w\in\RR$.}
\]
\par
We claim that~$F(w)$ is a strictly decreasing function. To see why, we note that
our earlier discussion shows that
\[
F(w) = f\bigl(w,U(w)\bigr) = U(w)^{-1}+U(w)^2-wU(w),
\]
where~$u=U(w)$ is the unique real solution to the equation
\[
\frac{\partial f}{\partial u}(w,u) = -u^{-2} + 2u - w = 0.
\]
Hence
\begin{align*}
\frac{dF}{dw}
&= \frac{d\phantom w}{dw}f\bigl(w,U(w)\bigr) \\
& = \frac{\partial f}{\partial w}\bigl(w,U(w)\bigr)
+ \underbrace{ \frac{\partial f}{\partial u}\bigl(w,U(w)\bigr) }_{\text{this is 0}}\cdot \frac{dU}{dw}(w) \\
&= -U(w) < 0.
\end{align*}
\par
Returning to our earlier calculation and using the
assumption~\eqref{eqn:n2gebetaalpha} that~$\g\le1$, we find that
\begin{align*}
  \a e_0 + \b \sum_{i=1}^r \frac{1}{e_i}
  &\ge (\a^2\b n)^{1/3} \inf_{u>\g} \left\{ u^{-1} + u^2 - \g u \right\}
  \quad\text{from~\eqref{eqn:ae0bsumi1rinf},} \\
  &\ge (\a^2\b n)^{1/3} F(\g)
  \quad\text{by defintiion of $F(w)$,} \\
  &\ge (\a^2\b n)^{1/3} F(1)
  \quad\begin{tabular}[t]{l} for all $0\le\g\le1$, since $F(w)$\\ is a decreasing function,\\ \end{tabular} \\
  &= (\a^2\b n)^{1/3}
  \quad\text{since it is easy to compute $F(1)=1$.}
\end{align*}
This completes the proof of Lemma~\ref{lemma:holderineq}.
\end{proof}


\begin{acknowledgement}
The authors would like to thank Dan Abramovich, Matt Baker, and David
Grant for their helpful advice.
\end{acknowledgement}



\bibliographystyle{plain}
\bibliography{ArithDyn,Intersection}

\def\cprime{$'$} \def\cprime{$'$}
\begin{thebibliography}{10}

\bibitem{MR1740514}
Francesco Amoroso and Roberto Dvornicich.
\newblock A lower bound for the height in abelian extensions.
\newblock {\em J. Number Theory}, 80(2):260--272, 2000.

\bibitem{MR591611}
M.~Anderson and David~W. Masser.
\newblock Lower bounds for heights on elliptic curves.
\newblock {\em Math. Z.}, 174(1):23--34, 1980.

\bibitem{MR1979685}
Matthew~H. Baker.
\newblock Lower bounds for the canonical height on elliptic curves over abelian
  extensions.
\newblock {\em Int. Math. Res. Not.}, (29):1571--1589, 2003.

\bibitem{MR2067482}
Matthew~H. Baker and Joseph~H. Silverman.
\newblock A lower bound for the canonical height on abelian varieties over
  abelian extensions.
\newblock {\em Math. Res. Lett.}, 11(2-3):377--396, 2004.

\bibitem{MR0296021}
P.~E. Blanksby and H.~L. Montgomery.
\newblock Algebraic integers near the unit circle.
\newblock {\em Acta Arith.}, 18:355--369, 1971.

\bibitem{MR1413570}
John~L. Boxall.
\newblock Une propri{\'e}t{\'e} des hauteurs locales de {N}{\'e}ron-tate sur
  les vari{\'e}t{\'e} ab{\'e}liennes.
\newblock {\em J. Th\'eor. Nombres Bordeaux}, 7(1):111--119, 1995.

\bibitem{MR1254751}
Sinnou David.
\newblock Minorations de hauteurs sur les vari\'{e}t\'{e}s ab\'{e}liennes.
\newblock {\em Bull. Soc. Math. France}, 121(4):509--544, 1993.

\bibitem{MR1799933}
Sinnou David and Marc Hindry.
\newblock Minoration de la hauteur de {N}\'{e}ron-{T}ate sur les
  vari\'{e}t\'{e}s ab\'{e}liennes de type {C}. {M}.
\newblock {\em J. Reine Angew. Math.}, 529:1--74, 2000.

\bibitem{MR1478502}
Sinnou David and Patrice Philippon.
\newblock Minorations des hauteurs normalis\'{e}es des sous-vari\'{e}t\'{e}s de
  vari\'{e}t\'{e}s ab\'{e}liennes.
\newblock In {\em Number theory ({T}iruchirapalli, 1996)}, volume 210 of {\em
  Contemp. Math.}, pages 333--364. Amer. Math. Soc., Providence, RI, 1998.

\bibitem{MR1949109}
Sinnou David and Patrice Philippon.
\newblock Minorations des hauteurs normalis\'{e}es des sous-vari\'{e}t\'{e}s de
  vari\'{e}t\'{e}s abeliennes. {II}.
\newblock {\em Comment. Math. Helv.}, 77(4):639--700, 2002.

\bibitem{MR2355454}
Sinnou David and Patrice Philippon.
\newblock Minorations des hauteurs normalis\'{e}es des sous-vari\'{e}t\'{e}s
  des puissances des courbes elliptiques.
\newblock {\em Int. Math. Res. Pap. IMRP}, (3):Art. ID rpm006, 113, 2007.

\bibitem{dobrowolski:lehmer}
E.~Dobrowolski.
\newblock On a question of {L}ehmer and the number of irreducible factors of a
  polynomial.
\newblock {\em Acta Arith.}, 34:391--401, 1979.

\bibitem{MR3598828}
Aur\'{e}lien Galateau and Val\'{e}ry Mah\'{e}.
\newblock Some consequences of {M}asser's counting theorem on elliptic curves.
\newblock {\em Math. Z.}, 285(1-2):613--629, 2017.

\bibitem{hindrynotesonlocalheights}
Marc Hindry.
\newblock Sur les hauteurs locales de {N}{\'e}ron sur les vari{\'e}t{\'e}s
  ab{\'e}liennes.
\newblock Pr{\'e}publications Math{\'e}matiques de l'U.R.A. 212 'Th{\'e}ories
  G{\'e}om{\'e}triques' no. 51, Universit{\'e} Paris 7 1993.

\bibitem{hindrysilverman:integralpts}
Marc Hindry and Joseph~H. Silverman.
\newblock The canonical height and integral points on elliptic curves.
\newblock {\em Invent. Math.}, 93(2):419--450, 1988.

\bibitem{hindrysilverman:lehmer}
Marc Hindry and Joseph~H. Silverman.
\newblock On {L}ehmer's conjecture for elliptic curves.
\newblock In {\em S\'eminaire de Th\'eorie des Nombres, Paris 1988--1989},
  volume~91 of {\em Progr. Math.}, pages 103--116. Birkh\"auser Boston, Boston,
  MA, 1990.

\bibitem{laurent:lehmer}
M.~Laurent.
\newblock Minoration de la hauteur de {N}\'eron-{T}ate.
\newblock In {\em S\'eminaire de Th\'eorie des Nombres}, Progress in
  Mathematics, pages 137--151. Birkh\"auser, 1983.
\newblock Paris 1981--1982.

\bibitem{MR1503118}
D.~H. Lehmer.
\newblock Factorization of certain cyclotomic functions.
\newblock {\em Ann. of Math. (2)}, 34(3):461--479, 1933.

\bibitem{MR766295}
D.~W. Masser.
\newblock Small values of the quadratic part of the {N}\'{e}ron-{T}ate height
  on an abelian variety.
\newblock {\em Compositio Math.}, 53(2):153--170, 1984.

\bibitem{masser:lehmer}
D.~W. Masser.
\newblock Counting points of small height on elliptic curves.
\newblock {\em Bull. Soc. Math. France}, 117(2):247--265, 1989.

\bibitem{MR3081000}
Fabien Pazuki.
\newblock Minoration de la hauteur de {N}\'{e}ron-{T}ate sur les surfaces
  ab\'{e}liennes.
\newblock {\em Manuscripta Math.}, 142(1-2):61--99, 2013.

\bibitem{MR2445828}
Nicolas Ratazzi.
\newblock Intersection de courbes et de sous-groupes et probl\`emes de
  minoration de hauteur dans les vari\'{e}t\'{e}s ab\'{e}liennes {C}.{M}.
\newblock {\em Ann. Inst. Fourier (Grenoble)}, 58(5):1575--1633, 2008.

\bibitem{MR747871}
Joseph~H. Silverman.
\newblock Lower bounds for height functions.
\newblock {\em Duke Math. J.}, 51(2):395--403, 1984.

\bibitem{MR2029512}
Joseph~H. Silverman.
\newblock A lower bound for the canonical height on elliptic curves over
  abelian extensions.
\newblock {\em J. Number Theory}, 104(2):353--372, 2004.

\bibitem{smyth:lehmer}
C.~J. Smyth.
\newblock On the product of the conjugates outside the unit circle of an
  algebraic integer.
\newblock {\em Bull. London Math. Soc.}, 3:169--175, 1971.

\bibitem{MR507748}
Cameron~L. Stewart.
\newblock Algebraic integers whose conjugates lie near the unit circle.
\newblock {\em Bull. Soc. Math. France}, 106(2):169--176, 1978.

\bibitem{MR1609514}
Emmanuel Ullmo.
\newblock Positivit\'e et discr\'etion des points alg\'ebriques des courbes.
\newblock {\em Ann. of Math. (2)}, 147(1):167--179, 1998.

\bibitem{VergerGaugrysurvey}
Jean-Louis Verger-Gaugry.
\newblock A survey on the conjecture of {L}ehmer and the conjecture of
  {S}chinzel--{Z}assenhaus, 2019.
\newblock \url{hal.archives-ouvertes.fr/hal-02315014}.

\bibitem{MR1418354}
Annette Werner.
\newblock Local heights on {M}umford curves.
\newblock {\em Math. Ann.}, 306(4):819--831, 1996.

\bibitem{MR1458753}
Annette Werner.
\newblock Local heights on abelian varieties with split multiplicative
  reduction.
\newblock {\em Compositio Math.}, 107(3):289--317, 1997.

\bibitem{MR1662481}
Annette Werner.
\newblock Local heights on abelian varieties and rigid analytic uniformization.
\newblock {\em Doc. Math.}, 3:301--319, 1998.

\bibitem{MR1609518}
Shou-Wu Zhang.
\newblock Equidistribution of small points on abelian varieties.
\newblock {\em Ann. of Math. (2)}, 147(1):159--165, 1998.

\bibitem{zhang1989unpublished}
Shouwu Zhang.
\newblock Lower bounds for heights on elliptic curves, June 1989.
\newblock \url{unpublished}.

\end{thebibliography}

\appendix

\section{Verification of some basic formulas}
\label{section:verifyformulas}

\begin{proposition}
\label{proposition:qprop1}
Let~$\bfn\in\ZZ^g$ and~$\bfu\in\GG_m^g(K_v)$.
\begin{parts}
\vspace{2\jot}
\Part{(a)}
$\displaystyle
\Theta\bigl(\bfu\cdot (\bfq\star2\bfn),\bfq\bigr)
=
({}^t\bfn\star\bfq\star\bfn)^{-1} ({}^t\bfn\star\bfu)^{-1} \Theta(\bfu,\bfq).
$
\vspace{2\jot}
\Part{(b)}
$\displaystyle
v\Bigl(\Theta\bigl(\bfu\cdot (\bfq\star2\bfn),\bfq\bigr)\Bigr)
=
v\Bigl(\Theta(\bfu,\bfq)\Bigr)
- {}^t\bfn{Q}\bfn-{}^t\bfn{v(\bfu)}.
$
\end{parts}
\end{proposition}
\begin{proof}
(a)\enspace
We compute
\begin{align*}
  \Theta\bigl(\bfu\cdot (\bfq\star2\bfn),\bfq\bigr)
  &= \sum_{\bfm\in\ZZ^g} ({}^t\bfm\star\bfq\star\bfm)
  \Bigl({}^t\bfm\star\bigl(\bfu\cdot (\bfq\star2\bfn)\bigr)\Bigr) \\
  &= \sum_{\bfm\in\ZZ^g} ({}^t\bfm\star\bfq\star\bfm)
  ({}^t\bfm\star\bfu) 
  \bigl({}^t\bfm\star\bfq\star2\bfn) \\
  &= \sum_{\bfm\in\ZZ^g} \Bigl({}^t\bfm\star\bfq\star(\bfm+2\bfn)\Bigr)
  ({}^t\bfm\star\bfu) \\
  &= \smash[b]{ \sum_{\bfm\in\ZZ^g} } \Bigl({}^t(\bfm+\bfn)\star\bfq\star(\bfm+\bfn)\Bigr)
  ({}^t\bfn\star\bfq\star\bfn)^{-1} \\
  &\omit\hfill$\displaystyle \Bigl({}^t(\bfm+\bfn)\star\bfu\Bigr) ({}^t\bfn\star\bfu)^{-1}$ \\
  &= ({}^t\bfn\star\bfq\star\bfn)^{-1} ({}^t\bfn\star\bfu)^{-1} \Theta(\bfu,\bfq).
\end{align*}
\par\noindent(b)\enspace
We have elementary formulas
\begin{equation}
  \label{eqn:threevformulas}
  v(\bfq\star\bfn)=Q\bfn,\quad
  v({}^t\bfn\star\bfq\star\bfn)={}^t\bfn{Q}\bfn,\quad
  v({}^t\bfn\star\bfu)={}^t\bfn{v(\bfu)}.
\end{equation}
We verify the first of these and leave the others to the reader.
\[
v(\bfq\star\bfn) =
\begin{pmatrix}
  v(q_{11}^{n_1}\cdots q_{1g}^{n_g}) \\
  \vdots \\
  v(q_{g1}^{n_1}\cdots q_{gg}^{n_g}) \\
\end{pmatrix}
=
\sum_{j=1}^g 
  n_j
  \begin{pmatrix}
  v(q_{1j}) \\ \vdots \\ v(q_{gj}) \\
  \end{pmatrix}
  =
  Q\bfn.
\]
Then applying~$v$ to the formula in~(a) gives the stated result.
\end{proof}

\begin{proposition}
\label{proposition:qprop2}
The function
\begin{gather*}
\Lambda(\,\cdot\,,\bfq) : \GG_m^g(K_v) \longrightarrow \RR,\\
\Lambda(\bfu,\bfq) = v\bigl(\Theta(\bfu,\bfq)\bigr) + {\dfrac{1}{4}} {}^tv(\bfu) Q^{-1} v(\bfu),
\end{gather*}
is~$\Omega$-invariant, and hence descends to a function
\[
\Lambda(\,\cdot\,,\bfq) : A(K_v)\cong \GG_m^g(K_v)/\Omega \longrightarrow \RR.
\]
\end{proposition}
\begin{proof}
We use the elementary formulas~\eqref{eqn:threevformulas} to compute
what happens when we translate~$\bfu$ by an element of the lattice.
\begin{align*}
  \Lambda & \bigl(\bfu\cdot(\bfq\star2\bfn),\bfq\bigr)  - \Lambda(\bfu,\bfq) \\* 
  &=  \left\{ v\bigl(\Theta(\bfu\cdot(\bfq\star2\bfn),\bfq)\bigr) + {\dfrac{1}{4}} {}^tv\bigl(\bfu\cdot(\bfq\star2\bfn)\bigr) Q^{-1} v\bigl(\bfu\cdot(\bfq\star2\bfn)\bigr) \right\} \\*
  & \omit\hfill$\displaystyle{} -  \left\{ v\bigl(\Theta(\bfu,\bfq)\bigr) + {\dfrac{1}{4}} {}^tv(\bfu) Q^{-1} v(\bfu) \right\}$ \\
  &= \Bigl\{ v\bigl(\Theta(\bfu\cdot(\bfq\star2\bfn),\bfq)\bigr) - v\bigl(\Theta(\bfu,\bfq)\bigr) \Bigr\} \\*
  & \omit\hfill$\displaystyle{}+
  \left\{ {\dfrac{1}{4}} {}^tv\bigl(\bfu\cdot(\bfq\star2\bfn)\bigr) Q^{-1} v\bigl(\bfu\cdot(\bfq\star2\bfn)\bigr) 
  - {\dfrac{1}{4}} {}^tv(\bfu) Q^{-1} v(\bfu) \right\}$ \\
  &= v\Bigl(({}^t\bfn\star\bfq\star\bfn)^{-1} ({}^t\bfn\star\bfu)^{-1}  \Bigr) \\*
  & \omit\hfill$\displaystyle{}
  + {\dfrac{1}{4}} \Bigl\{
  {}^tv(\bfq\star2\bfn)Q^{-1}v(\bfu) 
  +
  {}^tv(\bfu)Q^{-1}v(\bfq\star2\bfn)
  $ \\
  & \omit\hfill$\displaystyle{}
  +
  {}^tv(\bfq\star2\bfn)Q^{-1}v(\bfq\star2\bfn)
  \Bigr\}
  $ \\
  &= \Bigl\{ -{}^t\bfn{Q}\bfn-{}^t\bfn v(\bfu) \Bigr\} \\*
  & \omit\hfill$\displaystyle{}
  +
  {\dfrac{1}{4}} \Bigl\{
      {}^t(2Q\bfn) Q^{-1} v(\bfu)
      + {}^tv(\bfu)Q^{-1} 2Q\bfn
      + {}^t(2Q\bfn)Q^{-1}(2Q\bfn)
      \Bigr\}
  $\\
  & = 0 \quad\text{since ${}^tQ=Q$.}     
\end{align*}
\end{proof}

\end{document}